\documentclass[a4paper,12pt]{article}

\usepackage{amsmath}
\usepackage{amsthm}
\usepackage{enumerate}
\usepackage{amssymb,amsfonts,latexsym,mathtools}
\usepackage{mathrsfs}
\usepackage{bm}
\usepackage{cases}
\usepackage{color}

\newcommand{\deq}{\stackrel{d}{=}}

\newcommand{\R}{\mathbb{R}}

\newcommand{\N}{\mathbb{N}}

\newcommand{\Z}{\mathbb{Z}}
\newcommand{\Q}{\mathbb{Q}}
\newcommand{\e}{\varepsilon}

\renewcommand{\P}{\mathbb{P}}
\usepackage{multirow}

\usepackage{url}

\usepackage{tikz}

\numberwithin{equation}{section}

\makeatletter
\renewcommand\section{\@startsection {section}{1}{\z@}%
{-3.5ex \@plus -1ex \@minus -.2ex}%
{2.3ex \@plus.2ex}%
{\normalfont\large\bf}}
\makeatother

\makeatletter
\renewcommand\subsection{\@startsection {subsection}{1}{\z@}%
{-3.5ex \@plus -1ex \@minus -.2ex}%
{2.3ex \@plus.2ex}%
{\normalfont\normalsize\bf}}
\makeatother

\theoremstyle{plain}
\newtheorem{thm}{Theorem}[section]
\newtheorem{lem}[thm]{Lemma}

\newtheorem{prop}[thm]{Proposition}
\theoremstyle{definition}
\newtheorem{Rem}[thm]{Remark}

\allowdisplaybreaks

\usepackage{hyperref}
\hypersetup{
setpagesize=false,
 bookmarksnumbered=true,
 bookmarksopen=true,
 colorlinks=true,
 linkcolor=blue,
 citecolor=red,
}

\begin{document}
\begin{center}
\Large \textbf{Elephant Random Walk Conditioned to Avoid Zero}
\end{center}
\begin{center}
Kohki Iba\footnote{
\begin{tabular}[t]{@{}l@{}}
Affiliation: Graduate School of Science, The University of Osaka, Osaka, Japan.\\
E-mail: \url{kohki.iba@gmail.com}
\end{tabular}
} and Go Tokumitsu\footnote{
\begin{tabular}[t]{@{}l@{}}
Affiliation: Graduate School of Science, The University of Osaka, Osaka, Japan.\\
E-mail: \url{go.tokumitsu@gmail.com}
\end{tabular}
}
\end{center}
\begin{abstract}
We investigate the long-time limit of an elephant random walk conditioned to avoid zero and show that the resulting process admits a Skorokhod-type embedding into a three-dimensional Bessel process. Using this embedding, we further establish several limit theorems for the conditioned elephant random walk, including a scaling limit to a deterministically rescaled and time-changed three-dimensional Bessel process, and a law of the iterated logarithm.
\end{abstract}


\section{Introduction}
\subsection{The Elephant Random Walk}
The \emph{elephant random walk (ERW)}, introduced by Sch\"{u}tz and Trimper \cite{Schutz-Trimper}, is a stochastic process whose transition probabilities depend on its past trajectory. Let
\begin{align*}
S_0:=0,\qquad S_n:=X_1+\cdots+X_n\ \text{for}\ n\ge 1.
\end{align*}
The process $(S_n)_{n=0,1,2,...}$ is called an ERW if its sequence of increments $(X_n)_{n=1,2,...}$ is determined according to the following rule:
\begin{align*}
X_1&=
\begin{cases}
+1 & \text{with probability }q,\\
-1 & \text{with probability }1-q,
\end{cases}\\
X_{n+1}&=
\begin{cases}
+X_{U_n} & \text{with probability }p,\\
-X_{U_n} & \text{with probability }1-p,
\end{cases}
\end{align*}
where $p,q\in(0,1)$ and $U_n$ is uniformly distributed on $\{1,...,n\}.$ All random choices involved in this construction are assumed to be mutually independent. Intuitively, the parameter $p$ may be regarded as measuring the strength of the memory. Indeed, when $p$ is close to $1$, the walk is more likely to move in the same direction as a randomly selected past increment, whereas when $p$ is close to $0$, it is more likely to move in the opposite direction. In particular, when $p=q=\frac{1}{2}$, the ERW reduces to the simple symmetric random walk. It is therefore natural to expect the asymptotic behavior of $(S_n)$ to depend strongly on the value of $p$. In fact, it is known that the behavior of $(S_n)$ undergoes a phase transition at $p=\frac{3}{4}$. For example, when $q\neq\frac12$, Sch\"{u}tz and Trimper \cite{Schutz-Trimper} showed that the mean of $S_n$ satisfies
\begin{align*}
\mathbb{P}[S_n]\sim\frac{2q-1}{\Gamma(2p)}n^{2p-1},
\end{align*}
where $\P[\cdot]$ denotes the expectation with respect to the distribution $\P$. Moreover, its second moment satisfies
\begin{align*}
\mathbb{\P}[S_n^2]
\sim
\begin{cases}
\displaystyle \frac{n}{3-4p} & \text{for}\ 0<p<\frac{3}{4},\\[3mm]
n\log n & \text{for}\ p=\frac{3}{4},\\[3mm]
\displaystyle \frac{n^{4p-2}}{(4p-3)\Gamma(4p-2)} & \text{for}\ \frac{3}{4}<p<1.
\end{cases}
\end{align*}
Here, $x_n\sim y_n$ means that 
\begin{align*}
\frac{x_n}{y_n}\longrightarrow 1 \qquad\text{as }n\to\infty.
\end{align*}
Accordingly, the regimes $0<p<\frac{3}{4}$, $p=\frac{3}{4}$, and $\frac{3}{4}<p<1$ are referred to as the \emph{diffusive}, \emph{critical}, and \emph{superdiffusive} regimes, respectively. Furthermore, Coletti and Papageorgiou \cite{Coletti-Papageorgiou} proved that $(S_n)$ is recurrent when $0<p\le\frac{3}{4}$ and transient when $\frac{3}{4}<p<1$. For comparison with other limit theorems, including scaling limits, see Remark \ref{ERW-limit-theorems}.

Since the transition mechanism of the ERW is defined in terms of its past trajectory, the process may at first appear to have a non-Markovian structure. However, $(S_n)$ can in fact be described as a time-inhomogeneous Markov chain. More precisely,
\begin{align}
\label{ERW-transitionprobability}
\P\Big(S_{n+1}=x\pm 1\mid S_n=x\Big)=\frac{1}{2}\pm \left(p-\frac{1}{2}\right)\frac{x}{n}\qquad \text{for}\ x\in \Z.
\end{align}
Here, $\P$ denotes the law of the ERW. In particular, the space-time process $((n,S_n))_{n=0,1,2,...}$ is a time-homogeneous Markov chain.

Martingale methods are frequently used in the analysis of ERWs. Define the sequence
\begin{align}
\label{sequence-an}
a_0:=0,\qquad a_n:=\frac{\Gamma(n)}{\Gamma(n+2p-1)}\ \text{for}\ n\ge 1.
\end{align}
It is known that
\begin{align}
M_n:=a_nS_n
\end{align}
is a martingale. Moreover, since $(M_n)$ is a binary splitting martingale, it can be represented as a time-changed Brownian motion. For further details, see the beginning of Subsection \ref{S2.1}.

\subsection{The ERW Conditioned to Avoid Zero}
For a Markov chain $(S_n)_{n=0,1,2,...}$, a conditioning problem is to study the long-time limit of the form
\begin{align}
\label{conditioning-limit}
\lim_{N\to \infty} \P_x(\Lambda\mid T_A>N)\qquad \text{for}\ \Lambda\in \sigma(S_0,...,S_m),\ m\in \N,
\end{align}
where $\P_x$ denotes the law of the original process started from $x$, and $T_A$ denotes the first hitting time of a set $A$, that is, $T_A:=\inf\{n\ge 1:\  S_n\in A\}$. We call this problem \emph{conditioning to avoid the set $A$} or \emph{conditioning to stay in the set $A^c$}.

The problem of conditioning a time-homogeneous Markov chain on a given event has a long history. Good \cite{Good} studied birth-death chains conditioned to stay positive, while Keener \cite{Keener} considered the simple random walk conditioned to stay positive (C-SRW), and Bertoin and Doney \cite{Bertoin-Doney} investigated real-valued random walks conditioned to stay positive. These conditioning results were subsequently extended to more general random walks. In this setting, Hambly, Kersting, and Kyprianou \cite{Hambly-Kersting-Kyprianou} established a law of the iterated logarithm (LIL), whereas Bryn-Jones and Doney \cite{Bryn-Jones-Doney} and Caravenna and Chaumont \cite{Caravenna-Chaumont} proved scaling-limit results. Although these results concern discrete-time, time-homogeneous Markov chains, there is also an extensive literature on conditioning continuous-time, time-homogeneous Markov processes. 

We can express the limit (\ref{conditioning-limit}) via Doob's $h$-transform. More precisely, let $h_A$ be a harmonic function for the process killed upon hitting $A$. Then the transition kernel obtained by the $h$-transform is given by
\begin{align*}
\frac{h_A(y)}{h_A(x)}\cdot P_n^A(x,dy)\qquad \text{for}\ x\in \{z:\ 0<h_A(z)<\infty\},
\end{align*}
where $P_n^A$ denotes the transition kernel of the process killed upon hitting $A$. Such a harmonic function $h_A$ is often constructed from the asymptotic behavior of the survival probability. More precisely, for a suitable normalizing sequence $(\rho(N))$, one has
\begin{align}
\label{limit-hittingtime}
\lim_{N\to \infty}\rho(N)\P_x(T_A>N)=h_A(x).
\end{align}
For further details on the $h$-transform, see, e.g., \cite{Chung-Walsh}. The conditioning limit described above is usually formulated for time-homogeneous Markov chains and therefore cannot be applied directly to the ERW, which is a time-inhomogeneous Markov chain. However, by considering the space-time process, which is a time-homogeneous Markov chain, one can develop an analogous conditioning procedure.

Henceforth, let $\P_{(m,y)}$ denote the law of the ERW started from $y\in \Z$ at time $m\ge 0$. For simplicity, we write $\P:=\P_{(0,0)}$. Throughout this paper, we assume that the ERW is recurrent, that is, $0<p\le \frac{3}{4}$. We define the first hitting time of $0$ by the ERW as
\begin{align*}
T_0:=\inf\{n\ge 1:\ S_n=0\}.
\end{align*}
For the sequence $(a_n)$, define
\begin{align*}
A_n:=\sum_{j=1}^n a_j^2.
\end{align*}
By Stirling's formula, we have
\begin{align}\label{an-asymp}
a_n\sim n^{1-2p}\qquad \text{as}\ n\to \infty
\end{align}
and hence
\begin{align}
\label{An-asymp}
A_n\sim \begin{cases}
\displaystyle \frac{n^{3-4p}}{3-4p}&(0<p<\frac{3}{4}),\\[3mm]
\log n&(p=\frac{3}{4}).
\end{cases}
\end{align}
The asymptotic behaviors of the survival probabilities for recurrent ERWs starting from $0$ are known as follows:

\begin{thm}[Bertoin {\cite[Theorem 4.1]{Bertoin-zero}} and Fang {\cite[{Theorem 1.2}]{Fang}}]
The following assertions hold:
\begin{enumerate}
\item If $0<p<\frac{3}{4}$, then for any $m=0,2,4,...$
\begin{align}
\label{Bertoin-lim}
\lim_{N\to \infty}\sqrt{A_{N+m}-A_m}\P_{(m,0)}(T_0>N)=\sqrt{\frac{2}{\pi}}a_{m+1}.
\end{align}
\item If $p=\frac{3}{4}$, then
\begin{align*}
\lim_{N\to \infty}\sqrt{A_N-A_0}\P_{(0,0)}(T_0>N)=\sqrt{\frac{2}{\pi}}a_1=\frac{2\sqrt{2}}{\pi}.
\end{align*}
\end{enumerate}
\end{thm}

In this paper, we extend these results to survival probabilities with an arbitrary starting point.

\begin{thm}
\label{Prop2.4}
Let $y\in \Z\setminus \{0\}$. Assume that $\P(S_m=y)>0$. Then we have
\begin{align}
\label{Prop2.4-main}
\lim_{N\to \infty}\sqrt{A_{m+N}-A_m}\P_{(m,y)}(T_0>N)=\sqrt{\frac{2}{\pi}}|y|a_m.
\end{align}
\end{thm}

The proof of this theorem will be given in Section \ref{S2}. Using this theorem, we obtain the following conditioning limit.

\begin{thm}
\label{mainthm-1}
Suppose that $0<p\le \frac{3}{4}$. Then we have
\begin{align}
\label{mainthm-eq1}
\lim_{N\to \infty}\P(A\mid T_0>N)=\P\left[1_A\cdot \frac{h(m,S_m)}{h(0,0)}1_{\{T_0>m\}}\right],
\end{align}
for $A\in \sigma(S_0,...,S_m)$, where
\begin{align*}
h(0,0)&:=a_1,\\
h(m,y)&:=a_m|y|\qquad \text{for}\ m>0,\ y\neq 0.
\end{align*}
Moreover, $h(m,y)$ is a space-time harmonic function for the ERW killed upon hitting $0$, that is,
\begin{align}
\label{mainthm-eq2}
\P_{(m,y)}\Big[h(m+1,S_{m+1}),\ S_{m+1}\neq 0\Big]=h(m,y),
\end{align}
for $(m,y)=(0,0),\ \text{or for}\ m\ge 1\ \text{and}\ y\neq 0$.
\end{thm}

The proof of this theorem will be given in Section \ref{S2}. We denote by $\Q$ the law of the process obtained through this conditioning limit. Namely, $\Q$ is the measure satisfying
\begin{align}
\label{h-transform}
\frac{d\Q}{d\P}\Big|_{\sigma(S_0,...,S_m)}=\frac{h(m,S_m)}{h(0,0)}1_{\{T_0>m\}}\qquad \text{for}\ m\ge 0.
\end{align}
The existence of such a measure $\Q$ follows from the fact that $(h(m,S_m)1_{\{T_0>m\}})_{m=0,1,2,...}$ is a martingale under $\P$. For further details, see Yano \cite[Theorem A.1]{Yano}. We call the process $(S_n)$ under $\Q$ the \emph{ERW conditioned to avoid zero (C-ERW)}. The transition probabilities of the C-ERW are given as follows.

\begin{thm}
\label{mainthm-2}
Suppose that $0<p\le \frac{3}{4}$. Then,
\begin{align}
\label{mainthm2-eq1}
\Q(S_0=0)&=1,\\
\label{mainthm2-eq2}
\Q(S_1=\pm 1)&=\P(S_1=\pm 1),
\end{align}
and, for $m\ge 1\ \text{and}\ x\neq 0$,
\begin{align*}
\label{mainthm2-eq3}
\Q(S_{m+1}=x\pm 1\mid S_m=x)&=\frac{m}{m+2p-1}\frac{|x\pm 1|}{|x|}\P(S_{m+1}=x\pm 1\mid S_m=x)\\
&=\frac{m}{m+2p-1}\frac{|x\pm 1|}{|x|}\left\{\frac{1}{2}\pm \left(p-\frac{1}{2}\right)\frac{x}{m}\right\}.
\stepcounter{equation}\tag{\theequation}
\end{align*}
\end{thm}

The proof of this theorem will be given in Section \ref{S2}.

\begin{Rem}
Setting $x=+1\ \text{or}\ x=-1$ in the above theorem, we obtain
\begin{align}
\label{CERW-when1}
\Q(S_{m+1}=0\mid S_m=1)=\Q(S_{m+1}=0\mid S_m=-1)=0.
\end{align}
Thus, under $\Q$, the process never hits $0$.
\end{Rem}

\begin{Rem}
When $p=\frac{1}{2}$, the transition probabilities of the C-ERW for $x>0$ coincide with those of the C-SRW introduced by Keener \cite{Keener}. For $x<0$, they coincide with those of the sign-reversed C-SRW.
\end{Rem}

\subsection{Embedding into the Three-Dimensional Bessel Process}
The Skorokhod embedding theorem ensures that a square-integrable martingale can be realized as a Brownian motion (BM) observed at a suitable sequence of stopping times. Hambly, Kersting, and Kyprianou \cite{Hambly-Kersting-Kyprianou} showed that the C-SRW can be embedded into a three-dimensional Bessel process (BES(3)). The appearance of the BES(3) in this setting is also natural in view of the classical result of McKean \cite{McKean} that BM conditioned to stay positive can be described as a BES(3). For further details on three-dimensional Bessel processes, see, e.g., \cite{Borodin-Salminen,Karatzas-Shreve}.

Let $(R_t)_{t\ge 0}$ be a BES(3), and let $\mathscr{P}_x$ denote the law of the BES(3) started from $x\ge 0$, that is, $R_0=x$, $\mathscr{P}_x$-a.s. We write $\mathscr{P}:=\mathscr{P}_0$ for simplicity. Define a sequence of stopping times $(\sigma_n)_{n=0,1,2,...}$ associated with $(R_t)$ by
\begin{align*}
\sigma_0&:=0,\\
\sigma_1&:=\inf\{t\ge 0:\ R_t=a_1\},\\
\sigma_{n+1}&:=\inf \left\{t>\sigma_n:\ R_t\notin \left(a_{n+1}\Big(\frac{R_{\sigma_n}}{a_n}-1\Big),a_{n+1}\Big(\frac{R_{\sigma_n}}{a_n}+1\Big)\right)\right\}\qquad \text{for}\ n\ge 1.
\end{align*}

The C-ERW then admits the following Skorokhod-type embedding into a BES(3).

\begin{prop}
\label{Prop1.6}
Suppose that $0<p\le \frac{3}{4}$. Then,
\begin{align}
\label{Bessel-embed}
(R_{\sigma_n})_{n=0,1,2,...}\ \text{under}\ \mathscr{P}\ \deq\ (a_n|S_n|)_{n=0,1,2,...}\ \text{under}\ \Q.
\end{align}
\end{prop}

This proposition will be proved in Section \ref{S3}. Moreover, the sequence of stopping times $(\sigma_n)$ satisfies the following asymptotic relation.

\begin{prop}\label{asym2}
For every $r>0$,
\begin{align*}
\lim_{n\to\infty}\frac{\sigma_n}{A_n}=1\qquad \mathscr{P}\text{-a.s. and in }L^r(\mathscr{P}).
\end{align*}
\end{prop}

This proposition will be proved in Section \ref{S3}.

\subsection{Properties of the ERW Conditioned to Avoid Zero}
Using the embedding into a BES(3) given in (\ref{Bessel-embed}), we obtain various limit theorems for the C-ERW. Throughout this paper, the symbol ``$\Longrightarrow$'' denotes convergence in distribution. We denote by $D([0,\infty))$ and $D((0,\infty))$ the Skorokhod spaces of c\`adl\`ag functions on $[0,\infty)$ and $(0,\infty)$, respectively, equipped with the Skorokhod topology.

\begin{thm}
\label{mainthm-properties-ERW}
The following assertions hold:
\begin{enumerate}
\item For $0<p\le \frac{3}{4}$, the process $(S_n)$ is transient under $\Q$. 
\item (Scaling limit, $p<\frac{3}{4}$)
\begin{align*}
\left(\frac{|S_{[nt]}|}{\sqrt{n}}\right)_{t>0}\ \text{under}\ \Q\implies \left(\frac{t^{2p-1}}{\sqrt{3-4p}}R_{t^{3-4p}}\right)_{t>0}\ \text{under}\ \mathscr{P}\quad\text{in $D((0,\infty))$}.
\end{align*}
\item (Scaling limit, $p=\frac{3}{4}$)
\begin{align*}
\left(\frac{|S_{[n^t]}|}{\sqrt{n^t\log n}}\right)_{t>0}\ \text{under}\ \Q\implies (R_t)_{t>0}\ \text{under}\ \mathscr{P}\quad\text{in $D((0,\infty))$}.
\end{align*}
\item (LIL, $p<\frac{3}{4}$)
\begin{align}\label{LIL1}
\limsup_{n\to\infty}\frac{|S_n|}{\sqrt{2n\log\log n}}=\sqrt{\frac{1}{3-4p}}\qquad\text{$\Q$-a.s.}
\end{align}
\item (LIL, $p=\frac{3}{4}$)
\begin{align}\label{LIL2}
\limsup_{n\to\infty}\frac{|S_n|}{\sqrt{2n\log n\log\log\log n}}=1\qquad\text{$\Q$-a.s.}
\end{align}
\item ($r$-th moment, $p<\frac{3}{4}$) For odd $r$, if $q=\frac12$, then $\Q[S_n^r]=0$ for every $n\ge0$. Otherwise, the following asymptotics hold:
\begin{align*}
\mathbb{Q}[S_n^r]\sim\begin{cases}
\displaystyle\frac{2^{k+1/2}(2q-1)}{\sqrt{\pi}}\cdot\frac{k!}{(3-4p)^{k-1/2}}n^{k-1/2}&\text{for}\ r=2k-1,\\[5mm]
\displaystyle(2k+1)!!\left(\frac{n}{3-4p}\right)^k&\text{for}\ r=2k.
\end{cases}
\end{align*}
\item ($r$-th moment, $p=\frac{3}{4}$) For odd $r$, if $q=\frac12$, then $\Q[S_n^r]=0$ for every $n\ge0$. Otherwise, the following asymptotics hold:
\begin{align*}
\mathbb{Q}[S_n^r]\sim\begin{cases}
\displaystyle\frac{2^{k+1/2}(2q-1)}{\sqrt{\pi}}k!(n\log n)^{k-1/2}&\text{for}\ r=2k-1,\\[5mm]
(2k+1)!!(n\log n)^{k}&\text{for}\ r=2k.
\end{cases}
\end{align*}
\end{enumerate}
\end{thm}
These results will be proved in Section \ref{S4}.
\begin{Rem}
\label{Rem-consequences-mainthm}
The LIL assertions above also imply almost sure convergence for the C-ERW,
with the same normalizations as those in
\cite[Theorems~3.1 and~3.4]{Ber}.

The absolute values in the scaling-limit assertions can be removed
by introducing a random sign.
Let $\varepsilon$ be independent of $R$ and satisfy
\[
\mathscr P(\varepsilon=1)=q,
\qquad
\mathscr P(\varepsilon=-1)=1-q.
\]
Then, for $0<p<\frac34$,
\begin{align*}
\left(
\frac{S_{[nt]}}{\sqrt n}
\right)_{t>0}
\Longrightarrow
\left(
\frac{\varepsilon t^{2p-1}}{\sqrt{3-4p}}
R_{t^{3-4p}}
\right)_{t>0}
\quad\text{in $D((0,\infty))$},
\end{align*}
and, for $p=\frac34$,
\begin{align*}
\left(
\frac{S_{[n^t]}}{\sqrt{n^t\log n}}
\right)_{t>0}
\Longrightarrow
(\varepsilon R_t)_{t>0}
\quad\text{in $D((0,\infty))$}.
\end{align*}
The latter assertions follow from
\eqref{eq-reflection-symmetry} and the fact that the C-ERW cannot
change its sign without visiting the origin.
\end{Rem}
\begin{Rem}
\label{ERW-limit-theorems}
For comparison with the limit theorems for the original ERW, we list below some known limit theorems for recurrent ERWs. Let $(B_t)_{t\ge 0}$ be a one-dimensional Brownian motion (BM), and let $\mathbf{P}_x$ denote the law of the BM started from $x\in \R$, that is, $B_0=x$, $\mathbf{P}_x$-a.s. We write $\mathbf{P}:= \mathbf{P}_0$ for simplicity.
\begin{enumerate}
\item[i.] \cite{Coletti-Papageorgiou} For $0<p\le \frac{3}{4}$, the process $(S_n)$ is recurrent under $\P$.
\item[ii.] \cite{Baur-Bertoin} (Scaling limit, $p<\frac{3}{4}$)
\begin{align*}
\left(\frac{S_{[nt]}}{\sqrt{n}}\right)_{t\ge 0}\ \text{under}\ \P\implies \left(\frac{t^{2p-1}}{\sqrt{3-4p}}B_{t^{3-4p}}\right)_{t\ge 0}\ \text{under}\ \mathbf{P}\quad\text{in $D([0,\infty))$}.
\end{align*}
\item[iii.] \cite{Baur-Bertoin} (Scaling limit, $p=\frac{3}{4}$)
\begin{align*}
\left(\frac{S_{[n^t]}}{\sqrt{n^t\log n}}\right)_{t\ge 0}\ \text{under}\ \P\implies (B_t)_{t\ge 0}\ \text{under}\ \mathbf{P}\quad\text{in $D([0,\infty))$}.
\end{align*}
\item[iv.] \cite{Ber,Col3} (LIL, $p<\frac{3}{4}$)
\begin{align*}
\limsup_{n\to\infty}\pm\frac{S_n}{\sqrt{2n\log\log n}}=\sqrt{\frac{1}{3-4p}}\qquad\text{$\P$-a.s.}
\end{align*}
\item[v.] \cite{Ber,Col3} (LIL, $p=\frac{3}{4}$)
\begin{align*}
\limsup_{n\to\infty}\pm\frac{S_n}{\sqrt{2n\log n\log\log\log n}}=1\qquad\text{$\P$-a.s.}
\end{align*}
\item[vi.] \cite{Hayashi-Oshiro-Takei} ($r$-th moment, $p<\frac{3}{4}$) For odd $r$, we assume that $q\neq\frac{1}{2}$.
\begin{align*}
\P[S_n^r]\sim\begin{cases}
\displaystyle \frac{2q-1}{\Gamma(2p)}\cdot\frac{(2k-1)!!}{(3-4p)^{k-1}}n^{k+2p-2}&\text{for}\ r=2k-1,\\[5mm]
\displaystyle (2k-1)!!\left(\frac{n}{3-4p}\right)^k&\text{for}\ r=2k.
\end{cases}
\end{align*}
\item[vii.] \cite{Hayashi-Oshiro-Takei} ($r$-th moment, $p=\frac{3}{4}$) For odd $r$, we assume that $q\neq\frac{1}{2}$.
\begin{align*}
\P[S_n^r]\sim\begin{cases}
\displaystyle\frac{2(2q-1)}{\sqrt{\pi}}(2k-1)!!n^{k-\frac{1}{2}}(\log n)^{k-1}&\text{for}\ r=2k-1,\\[5mm]
(2k-1)!!(n\log n)^k&\text{for}\ r=2k.
\end{cases}
\end{align*}
\end{enumerate}
Several similarities and differences emerge from this comparison.
Conditioning the ERW to avoid zero turns recurrence into transience.

On the other hand, the normalizing sequences in the LIL
remain unchanged. Moreover, the C-ERW has the same diffusion
classification as the ordinary ERW: it is diffusive for
$p<\frac34$ and marginally superdiffusive for $p=\frac34$.
The even moments also have the same orders, although their
asymptotic constants differ. In contrast, when $q\ne1/2$,
conditioning substantially changes the orders of the odd moments.
When $q=1/2$, all odd moments of both the ERW and the C-ERW are zero.
\end{Rem}

\begin{Rem}
For the simple random walk (SRW), it is well known from Donsker's invariance principle that the diffusively rescaled process converges in distribution to BM. Keener \cite{Keener} constructed the simple random walk conditioned to stay positive (C-SRW). Moreover, McKean \cite{McKean} showed that BM conditioned to stay positive can be described as a BES(3). Furthermore, Bryn-Jones and Doney \cite{Bryn-Jones-Doney} proved that the scaling limit of the C-SRW is given by a BES(3).

In the present paper, we establish ERW analogues of these results. Let $0<p<\frac{3}{4}$. Baur and Bertoin \cite{Baur-Bertoin} showed that the scaling limit of the ERW is given by
\begin{align*}
\left(\widehat{B}_t:=\frac{t^{2p-1}}{\sqrt{3-4p}}B_{t^{3-4p}}\right)_{t\ge 0}.
\end{align*}
The limiting process $(\widehat{B}_t)$ is a Gaussian process called the \emph{noise reinforced Brownian motion (NR-BM)}, and is obtained from a BM by a deterministic time change and rescaling. Theorem \ref{mainthm-1} constructs the C-ERW, while Theorem \ref{mainthm-properties-ERW} identifies its scaling limit as
\begin{align*}
\left(\widehat{R}_t:=\frac{t^{2p-1}}{\sqrt{3-4p}}R_{t^{3-4p}}\right)_{t\ge 0}.
\end{align*}
The limiting process $(\widehat{R}_t)$ is called the \emph{noise reinforced three-dimensional Bessel process (NR-BES(3))} in Bertoin \cite{Bertoin-NRBP}, and is obtained from a BES(3) by a deterministic time change and rescaling. Since the NR-BM and the NR-BES(3) are obtained by applying the same deterministic time change and rescaling to a BM and a BES(3), respectively, it is natural to regard the NR-BES(3) as the process corresponding to the NR-BM conditioned to stay positive. The following diagram summarizes these relationships:

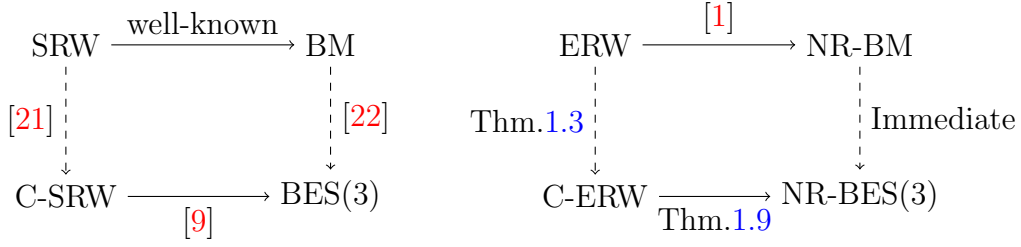
\begin{center}
\begin{tikzpicture}
\node (SRW) at (0,0) {SRW};
\node (BM) at (3.5,0) {BM};
\node (ERW) at (7,0) {ERW};
\node (NRBM) at (10.5,0) {NR-BM};
\node (CSRW) at (0,-2) {C-SRW};
\node (BP) at (3.5,-2) {BES(3)};
\node (CERW) at (7,-2) {C-ERW};
\node (NRBP) at (10.5,-2) {NR-BES(3)};
\draw[->] (SRW) -- node[above] {well-known} (BM);
\draw[->, dashed] (BM) -- node[right] {\cite{McKean}} (BP);
\draw[->, dashed] (SRW) -- node[left] {\cite{Keener}} (CSRW);
\draw[->] (CSRW) -- node[below] {\cite{Bryn-Jones-Doney}} (BP);
\draw[->] (ERW) -- node[above] {\cite{Baur-Bertoin}} (NRBM);
\draw[->, dashed] (NRBM) -- node[right] {Immediate} (NRBP);
\draw[->, dashed] (ERW) -- node[left] {Thm.\ref{mainthm-1}} (CERW);
\draw[->] (CERW) -- node[below] {Thm.\ref{mainthm-properties-ERW}} (NRBP);
\end{tikzpicture}
\end{center}

Here, solid arrows represent scaling limits, whereas dashed arrows represent conditioning.
\end{Rem}

\subsection*{Organization of the paper}
This paper is organized as follows. In Section \ref{S2}, we will prove Theorems \ref{Prop2.4}, \ref{mainthm-1}, and \ref{mainthm-2}. In Section \ref{S3}, we will prove Propositions \ref{Prop1.6} and \ref{asym2}. In Section \ref{S4}, we will prove Theorem \ref{mainthm-properties-ERW}. Finally, Appendix \ref{Appendix} contains the proof of the lemma omitted from the main text.


\section{The ERW Conditioned to Avoid Zero}
\label{S2}
In this section, we prove Theorems \ref{Prop2.4}, \ref{mainthm-1}, and \ref{mainthm-2}. We first treat the diffusive case $0<p<\frac{3}{4}$, where Bertoin's Brownian embedding can be used directly. The critical case $p=\frac{3}{4}$ is then handled by the same argument, with the corresponding critical estimates.

\subsection{The Diffusive Case}
\label{S2.1}
Throughout this subsection, we assume that $0<p<\frac{3}{4}$.

By reflection symmetry, changing the sign of the first step $X_1$
transforms the entire path $(S_n)_{n=0,1,2,\ldots}$ into its reflection
$(-S_n)_{n=0,1,2,\ldots}$. Consequently, the law of the zero set of $(S_n)$
does not depend on $q$. In particular, neither does the distribution
of $T_0$. Therefore, for the Brownian embedding argument below, we
may and do assume without loss of generality that
\begin{align*}
\mathbb P(X_1=1)=\mathbb P(X_1=-1)=\frac12.
\end{align*}
Any conclusion obtained from this argument concerning the zero set,
and in particular the distribution of $T_0$, remains valid for
arbitrary $q\in(0,1)$.

We first recall the explicit embedding of the ERW into Brownian
motion constructed by Bertoin \cite{Bertoin-zero}, following the Skorokhod embedding approach of Coletti, Gava, and Sch\"utz \cite{Col3}. Let $(B_t)_{t\ge 0}$ be a Brownian motion, and let $\mathbf{P}_x$ denote the law of a Brownian motion started at $x\in \R$, so that $\mathbf{P}_x(B_0=x)=1$. We write $\mathbf{P}:=\mathbf{P}_0$ for simplicity. For $m\ge 0$, we define inductively the increasing sequence of stopping times $(\tau_{m,n})_{n\ge 0}$ by
\begin{align*}
\tau_{m,0}&:=0,\\
\tau_{m,n+1}&:=\inf\left\{t>\tau_{m,n}:\ B_t-B_{\tau_{m,n}}=\frac{1-2p}{m+n+2p-1}B_{\tau_{m,n}}\pm a_{m+n+1}\right\},
\end{align*}
where, when $m=n=0$ and $p=\frac12$, the fraction above is understood to be 0. Intuitively, the stopping time $\tau_{m,2k}$ is the first time after $\tau_{m,2k-1}$ at which the Brownian motion hits the set
\begin{align*}
\Big\{0,\ \pm 2a_{m+2k},\ \pm 4a_{m+2k},...,\ \pm 2k a_{m+2k}\Big\}.
\end{align*}
Similarly, the stopping time $\tau_{m,2k+1}$ is the first time after $\tau_{m,2k}$ at which the Brownian motion hits the set
\begin{align*}
\Big\{\pm a_{m+2k+1},\ \pm 3a_{m+2k+1},...,\ \pm (2k+1)a_{m+2k+1}\Big\}.
\end{align*}
In this setting, Bertoin \cite{Bertoin-zero} showed that
\begin{align}
\label{Bertoin-embed}
(B_{\tau_{m,n}})_{n=0,1,2,...}\ \text{under}\ \mathbf{P}\ \deq\ (a_{m+n} S_{m+n})_{n=0,1,2,...}\ \text{under}\ \P_{(m,0)}.
\end{align}
Using this embedding, Bertoin \cite{Bertoin-zero} obtained the limit (\ref{Bertoin-lim}) of the tail distribution of $T_0$ under $\P_{(m,0)}$. To extend this result under $\P_{(m,y)}$ for $y\neq 0$, we consider the following sequence of stopping times. For $N\ge 1$, we define
\begin{align*}
D_N^{(m,y)}:=\Big\{(y+N- 2l)a_{m+N}:\ l=0,1,...,N\Big\}.
\end{align*}
We denote by $\tau_A$ the first hitting time of a set $A$ by the Brownian motion $(B_t)$, that is, we define
\begin{align*}
\tau_A:=\inf\{t\ge 0:\ B_t\in A\}.
\end{align*}
Using these sets, we define a sequence of stopping times $(\tau^{(N)})$ by
\begin{align*}
\tau^{(0)}&:=0,\\
\tau^{(N)}&:=\tau^{(N-1)}+\tau_{D_{N}^{(m,y)}}\circ \theta_{\tau^{(N-1)}}\qquad \text{for}\ N\ge 1,
\end{align*}
where $\theta$ denotes the shift operator of the Brownian motion.

The following lemma is the key link between the survival probability of the ERW and the Brownian embedding.

\begin{lem}
\label{Lem2.1}
Let $y\in \Z\setminus \{0\}$. Assume that $\P(S_m=y)>0$. Then, for $N\ge 0$, we have
\begin{align}
\label{Lem2.1-main}
\P_{(m,y)}(T_0>N)=\mathbf{P}_{ya_m}(\tau_{\{0\}}>\tau^{(N)}).
\end{align}
\end{lem}

\begin{proof}
Since, under $\P_{(m,y)}$, the ERW $(S_n)$ is at $y$ at time $m$, the possible values of the last time before $m$ at which $(S_n)$ is at 0 are
\begin{align*}
k_i:=m-|y|-2i\qquad \text{for}\ i=0,1,2,...,\ \frac{m-|y|}{2}.
\end{align*}
We define
\begin{align*}
n_i:=|y|+2i\qquad \text{for}\ i=0,1,2,...,\ \frac{m-|y|}{2}.
\end{align*}
Since $m=n_i+k_i$ for such $i$, we have
\begin{align*}
\label{Lem2.1-eq1}
\P_{(m,y)}(T_0>N)&=\P(T_0\circ \vartheta_m>N\mid S_m=y)\\
&=\frac{\P(T_0\circ \vartheta_m>N,\ S_m=y)}{\P(S_m=y)}\\
&=\frac{\sum_{i}\P(T_0\circ \vartheta_{n_i}\circ\vartheta_{k_i}>N,\ S_{k_i}=0,\ T_0\circ \vartheta_{k_i}>n_i,\ S_{k_i+n_i}=y)}{\sum_{i}\P(T_0\circ \vartheta_{k_i}>n_i,\ S_{k_i}=0,\ S_{k_i+n_i}=y)}\\
&=\frac{\sum_{i}\P(T_0\circ \vartheta_{k_i}>N+n_i,\ S_{k_i}=0,\ S_{k_i+n_i}=y)}{\sum_{i}\P(T_0\circ \vartheta_{k_i}>n_i,\ S_{k_i}=0,\ S_{k_i+n_i}=y)},
\stepcounter{equation}\tag{\theequation}
\end{align*}
where $\vartheta$ denotes the shift operator of the ERW. We define the stopping time $\tau_{m}^\ast$ by
\begin{align*}
\tau_{m}^\ast:&=\inf \{\tau_{m,n}:\ n\ge 1,\ B_{\tau_{m,n}}=0\}\\
&=\inf\{t>\tau_{m,1}:\ B_t=0\},
\stepcounter{equation}\tag{\theequation}
\end{align*}
where the second equality follows from the argument in Bertoin \cite[p. 5553]{Bertoin-zero}. By (\ref{Bertoin-embed}), the numerator of the right-hand side of (\ref{Lem2.1-eq1}) is
\begin{align*}
&\P(T_0\circ \vartheta_{k_i}>N+n_i,\ S_{k_i}=0,\ S_{k_i+n_i}=y)\\
  &\qquad =\P(S_{k_i}=0)\P_{(k_i,0)}(T_0\circ \vartheta_{k_i}>N+n_i,\ S_{k_i+n_i}=y)\\
  &\qquad=\P(S_{k_i}=0)\mathbf{P}(\tau_{k_i}^\ast>\tau_{k_i,N+n_i},\ B_{\tau_{k_i,n_i}}=ya_m)\\
&\qquad =\P(S_{k_i}=0)\mathbf{P}(\tau_{k_i}^\ast>\tau_{k_i,n_i},\ B_{\tau_{k_i,n_i}}=ya_m)\mathbf{P}_{ya_m}(\tau_{\{0\}}>\tau^{(N)}).
\stepcounter{equation}\tag{\theequation}
\end{align*}
Similarly, the denominator of the right-hand side of (\ref{Lem2.1-eq1}) is
\begin{align*}
\P(S_{k_i}=0,\ T_0\circ \vartheta_{k_i}>n_i,\ S_{k_i+n_i}=y)&=\P(S_{k_i}=0)\mathbf{P}(\tau_{k_i}^\ast>\tau_{k_i,n_i},\ B_{\tau_{k_i,n_i}}=ya_m).
\end{align*}
Therefore, by canceling the common factors on the right-hand side of (\ref{Lem2.1-eq1}), we obtain (\ref{Lem2.1-main}). This completes the proof.
\end{proof}

We next consider the analogue of \cite[Proposition 4.2]{Bertoin-zero} for $\tau^{(N)}$.

\begin{lem}
\label{Lem2.2}
Let $y\in \Z\setminus \{0\}$. Assume that $\P(S_m=y)>0$. Then, for $\e>0$ and $r\ge 1$, there exists a constant $c_{\e,r,m,y}>0$ such that for any $N\ge 0$,
\begin{align}
\label{Lem2.2-main}
\mathbf{P}_{ya_m}\Big(|\tau^{(N)}-(A_{m+N}-A_m)|>\e (m+N)^{3-4p}\Big)\le c_{\e,r,m,y}(m+N)^{-r}.
\end{align}
\end{lem}

\begin{proof}
We define the following event:
\begin{align*}
E_{m,y}:=\Big\{B_{\tau_{k_0,0}}=0,\ B_{\tau_{k_0,1}}=\mathrm{sgn}(y)a_{k_0+1},\ ...,\ B_{\tau_{k_0,|y|}}=\mathrm{sgn}(y)|y|a_m\Big\}.
\end{align*}
Note that $\mathbf{P}(E_{m,y})>0$. Since 
\begin{align*}
\tau^{(N)}\circ \theta_{\tau_{k_0,|y|}}=\tau_{k_0,|y|+N}-\tau_{k_0,|y|}\qquad \text{on}\ E_{m,y}
\end{align*}
and $k_0=m-|y|$, we have 
\begin{align*}
&\Big|\tau^{(N)}\circ \theta_{\tau_{k_0,|y|}}-(A_{m+N}-A_m)\Big|\\
&\qquad =\Big|(\tau_{k_0,|y|+N}-\tau_{k_0,|y|})-(A_{k_0+|y|+N}-A_{k_0+|y|})\Big|\\
&\qquad \le \Big|\tau_{k_0,|y|+N}-(A_{k_0+|y|+N}-A_{k_0})\Big|+\Big|\tau_{k_0,|y|}-(A_{k_0+|y|}-A_{k_0})\Big|\\
&\qquad \le 2\sup_{l\le |y|+N}\Big|\tau_{k_0,l}-(A_{k_0+l}-A_{k_0})\Big|
\stepcounter{equation}\tag{\theequation}
\end{align*}
on $E_{m,y}$. Thus, by the strong Markov property and \cite[Proposition 4.2]{Bertoin-zero}, for $\e>0$ and $r\ge 1$, there exists a constant $c_{\frac{\e}{2},r}>0$ such that
\begin{align*}
&\mathbf{P}(E_{m,y})\mathbf{P}_{ya_m}\Big(|\tau^{(N)}-(A_{m+N}-A_m)|>\e (m+N)^{3-4p}\Big)\\
&\qquad =\mathbf{P}\Big(|\tau^{(N)}\circ \theta_{\tau_{k_0,|y|}}-(A_{m+N}-A_m)|>\e (k_0+|y|+N)^{3-4p},\ E_{m,y}\Big)\\
&\qquad \le \mathbf{P}\left(2\sup_{l\le |y|+N}|\tau_{k_0,l}-(A_{k_0+l}-A_{k_0})|\ge \e (k_0+|y|+N)^{3-4p}\right)\\
&\qquad \le c_{\frac{\e}{2},r}(k_0+|y|+N)^{-r}\\
&\qquad =c_{\frac{\e}{2},r}(m+N)^{-r}.
\stepcounter{equation}\tag{\theequation}
\end{align*}
Therefore, since $\mathbf{P}(E_{m,y})>0$, by setting 
\begin{align*}
c_{\e,r,m,y}:=\frac{c_{\frac{\e}{2},r}}{\mathbf{P}(E_{m,y})},
\end{align*}
we obtain (\ref{Lem2.2-main}). This completes the proof.
\end{proof}

Next, we compute the asymptotic behavior of the hitting time of $0$ for the Brownian motion. Although this result is well known, we give a proof for the sake of completeness.

\begin{lem}
\label{Lem2.3}
For $x\in \R\setminus \{0\}$, we have
\begin{align*}
\lim_{t\to \infty}\sqrt{t}\mathbf{P}_{x}(\tau_{\{0\}}>t)=\sqrt{\frac{2}{\pi}}|x|.
\end{align*}
\end{lem}
\begin{proof}
Let $x>0$. It is known that the hitting-time process $(\tau_{\{x\}})_{x\ge 0}$ of a Brownian motion is a $\frac{1}{2}$-stable subordinator (see, e.g., \cite[Theorem 7.4.1]{Durrett}). Thus, we have
\begin{align*}
\mathbf{P}_{x}(\tau_{\{0\}}>t)&=\mathbf{P}(\tau_{\{x\}}>t)\\
&=\int_t^\infty \frac{x}{\sqrt{2\pi}}e^{-\frac{x^2}{2s}}s^{-\frac{3}{2}}ds\\
&=\frac{x}{\sqrt{2\pi}}\cdot \frac{1}{\sqrt{t}}\int_1^\infty e^{-\frac{x^2}{2tu}}u^{-\frac{3}{2}}du\qquad (s=tu).
\stepcounter{equation}\tag{\theequation}
\end{align*}
Since $e^{-\frac{x^2}{2tu}}u^{-\frac{3}{2}}\le u^{-\frac{3}{2}}\in L^1(1,\infty)$, by the dominated convergence theorem, we have
\begin{align*}
\lim_{t\to \infty}\sqrt{t}\mathbf{P}_{x}(\tau_{\{0\}}>t)=\frac{x}{\sqrt{2\pi}}\int_1^\infty u^{-\frac{3}{2}}du=\sqrt{\frac{2}{\pi}}x.
\end{align*}
By the symmetry of Brownian motion, the case $x<0$ is treated in the same way. This completes the proof.
\end{proof}

Now let us prove Theorem \ref{Prop2.4}.

\begin{proof}[Proof of Theorem \ref{Prop2.4}]
By Lemma \ref{Lem2.1}, it suffices to prove the desired asymptotic for $\mathbf{P}_{ya_m}(\tau_{\{0\}}>\tau^{(N)})$. Fix an arbitrary $0<\delta<1$. Since
\begin{align*}
\Big\{\tau_{\{0\}}>(1+\delta)(A_{m+N}-A_m)\Big\}\cap \Big\{\tau^{(N)}\le (1+\delta)(A_{m+N}-A_m)\Big\}\subset \{\tau_{\{0\}}>\tau^{(N)}\},
\end{align*}
we have
\begin{align*}
\label{Prop2.4-eq0}
&\mathbf{P}_{ya_m}(\tau_{\{0\}}>\tau^{(N)})\\
&\qquad \ge \mathbf{P}_{ya_m}\Big(\tau_{\{0\}}>(1+\delta)(A_{m+N}-A_m)\Big)-\mathbf{P}_{ya_m}\Big(\tau^{(N)}>(1+\delta)(A_{m+N}-A_m)\Big).
\stepcounter{equation}\tag{\theequation}
\end{align*}
First, we consider the first term on the right-hand side of (\ref{Prop2.4-eq0}). By Lemma \ref{Lem2.3}, we have
\begin{align*}
\label{Prop2.4-eq1}
&\lim_{N\to \infty}\sqrt{A_{m+N}-A_m}\mathbf{P}_{ya_m}\Big(\tau_{\{0\}}>(1+\delta)(A_{m+N}-A_m)\Big)\\
&\qquad =\lim_{N\to \infty}\frac{1}{\sqrt{1+\delta}}\sqrt{(1+\delta)(A_{m+N}-A_m)}\mathbf{P}_{ya_m}\Big(\tau_{\{0\}}>(1+\delta)(A_{m+N}-A_m)\Big)\\
&\qquad =\frac{1}{\sqrt{1+\delta}}\sqrt{\frac{2}{\pi}}|y|a_m.
\stepcounter{equation}\tag{\theequation}
\end{align*}
Next, we consider the second term on the right-hand side of (\ref{Prop2.4-eq0}). By (\ref{An-asymp}), for all $\e>0$, there exists a constant $M>0$ such that
\begin{align*}
A_n-A_m>(1-\e)\frac{n^{3-4p}}{3-4p}\qquad\text{for}\ n>M.
\end{align*}
Thus, by Lemma \ref{Lem2.2}, for $N>M-m$, we have
\begin{align*}
&\mathbf{P}_{ya_m}\Big(\tau^{(N)}>(1+\delta)(A_{m+N}-A_m)\Big)\\
&\qquad =\bm{P}_{ya_m}\Big(\tau^{(N)}-(A_{m+N}-A_m)>\delta(A_{m+N}-A_m) \Big)\\
&\qquad \le \bm{P}_{ya_m}\left(|\tau^{(N)}-(A_{m+N}-A_m)|>\frac{1}{3-4p}\delta(1-\e)(m+N)^{3-4p} \right)\\
&\qquad \le c_{\e',r,m,y}(m+N)^{-r},
\stepcounter{equation}\tag{\theequation}
\end{align*}
where $\e':=\frac{1}{3-4p}\delta (1-\e).$ Thus, setting $r>\frac{1}{2}(3-4p)$, we have
\begin{align*}
\label{Prop2.4-eq2}
0\le &\limsup_{N\to \infty}\sqrt{A_{m+N}-A_m}\mathbf{P}_{ya_m}\Big(\tau^{(N)}>(1+\delta)(A_{m+N}-A_m)\Big)\\
&\qquad \le \limsup_{N\to \infty} c_{\e',r,m,y}\sqrt{A_{m+N}-A_m}(m+N)^{-r}\\
&\qquad =\limsup_{N\to \infty}\frac{c_{\e',r,m,y}}{\sqrt{3-4p}}(m+N)^{\frac{1}{2}(3-4p)}(m+N)^{-r}\\
&\qquad =0.
\stepcounter{equation}\tag{\theequation}
\end{align*}
Therefore, by (\ref{Prop2.4-eq1}) and (\ref{Prop2.4-eq2}), we have
\begin{align}
\label{Prop2.4-eq3}
\liminf_{N\to \infty}\sqrt{A_{m+N}-A_m}\mathbf{P}_{ya_m}(\tau_{\{0\}}>\tau^{(N)})\ge \frac{1}{\sqrt{1+\delta}}\sqrt{\frac{2}{\pi}}|y|a_m.
\end{align}

Next, we consider the upper bound. For the above $\delta$, we have
\begin{align*}
\{\tau_{\{0\}}>\tau^{(N)}\}\cap \Big\{\tau^{(N)}\ge (1-\delta)(A_{m+N}-A_m)\Big\}\subset \Big\{\tau_{\{0\}}>(1-\delta)(A_{m+N}-A_m)\Big\}.
\end{align*}
Thus, we have
\begin{align*}
&\mathbf{P}_{ya_m}(\tau_{\{0\}}>\tau^{(N)})\\
&\qquad \le \mathbf{P}_{ya_m}\Big(\tau_{\{0\}}> (1-\delta)(A_{m+N}-A_m)\Big)+\mathbf{P}_{ya_m}\Big(\tau^{(N)}<(1-\delta)(A_{m+N}-A_m)\Big).
\stepcounter{equation}\tag{\theequation}
\end{align*}
In the same way as in the first part of the argument, we obtain
\begin{align}
\label{Prop2.4-eq4}
\limsup_{N\to \infty}\sqrt{A_{m+N}-A_m}\mathbf{P}_{ya_m}(\tau_{\{0\}}>\tau^{(N)})\le \frac{1}{\sqrt{1-\delta}}\sqrt{\frac{2}{\pi}}|y|a_m.
\end{align}
Since $\delta$ is arbitrary, (\ref{Prop2.4-eq3}) and (\ref{Prop2.4-eq4}) imply (\ref{Prop2.4-main}). This completes the proof.
\end{proof}

Now let us prove Theorem \ref{mainthm-1}.

\begin{proof}[Proof of Theorem \ref{mainthm-1} for $0<p<\frac{3}{4}$]
For $A\in \sigma(S_0,...,S_m)$, by the Markov property of the ERW, we have
\begin{align*}
\P(A\mid T_0>m+N)&=\frac{\P(A,\ T_0>m+N)}{\P(T_0>m+N)}\\
&=\P\left[1_A\cdot \frac{\P_{(m,S_m)}(T_0>N)}{\P(T_0>m+N)}1_{\{T_0>m\}}\right]\\
&=\sum_{y}\P\left[1_A\cdot \frac{\P_{(m,y)}(T_0>N)}{\P(T_0>m+N)}1_{\{T_0>m\}},\ S_m=y\right].
\stepcounter{equation}\tag{\theequation}
\end{align*}
Here, note that the above summation is finite. Thus, by (\ref{Bertoin-lim}) and (\ref{Prop2.4-main}), we have
\begin{align*}
\lim_{N\to \infty}\P(A\mid T_0>m+N)&=\sum_y \P\left[1_A\cdot \frac{\sqrt{\frac{2}{\pi}}|y|a_m}{\sqrt{\frac{2}{\pi}}a_1}1_{\{T_0>m\}},\ S_m=y\right]\\
&=\P\left[1_A\cdot \frac{a_m|S_m|}{a_1}1_{\{T_0>m\}}\right].
\stepcounter{equation}\tag{\theequation}
\end{align*}
Thus, we have proved (\ref{mainthm-eq1}).

Next, we prove (\ref{mainthm-eq2}). By (\ref{ERW-transitionprobability}) and (\ref{sequence-an}), for $m\ge 1$ and $y\neq 0$, we have
\begin{align*}
&\P_{(m,y)}\Big[h(m+1,S_{m+1}),\ S_{m+1}\neq 0\Big]\\
&\qquad =\P_{(m,y)}\Big[a_{m+1}|S_{m+1}|,\ S_{m+1}\neq 0\Big]\\
&\qquad =a_{m+1}|y+1|\left\{\frac{1}{2}+ \left(p-\frac{1}{2}\right)\frac{y}{m}\right\}+a_{m+1}|y-1|\left\{\frac{1}{2}- \left(p-\frac{1}{2}\right)\frac{y}{m}\right\}\\
&\qquad =a_{m+1}\left\{\frac{1}{2}\Big(|y+1|+|y-1|\Big)+\left(p-\frac{1}{2}\right)\frac{y}{m}\Big(|y+1|-|y-1|\Big) \right\}\\
&\qquad =a_{m+1}\left\{|y|+\left(p-\frac{1}{2}\right)\frac{y}{m}\cdot 2\ \mathrm{sgn}(y) \right\}\\
&\qquad =\frac{\Gamma(m+1)}{\Gamma(m+2p)}\cdot \frac{m+2p-1}{m}|y|\\
&\qquad =\frac{\Gamma(m)}{\Gamma(m+2p-1)}|y|\\
&\qquad =a_m|y|\\
&\qquad =h(m,y).
\stepcounter{equation}\tag{\theequation}
\end{align*}
For $m=0$, we have
\begin{align*}
\P\Big[h(1,S_1),\ S_1\neq 0\Big]=a_1\P[|S_1|]=a_1=h(0,0).
\end{align*}
This completes the proof.
\end{proof}

Now let us prove Theorem \ref{mainthm-2}.

\begin{proof}[Proof of Theorem \ref{mainthm-2} for $0<p<\frac{3}{4}$]
Since (\ref{mainthm2-eq1}) and (\ref{mainthm2-eq2}) are immediate, it remains to prove (\ref{mainthm2-eq3}). Assume that $m\ge 1$ and $x\neq 0$. By (\ref{h-transform}) and the Markov property, we have
\begin{align*}
\label{Thm1.3-pfeq1}
&\Q(S_{m+1}=x\pm 1\mid S_m=x)\\
&\qquad =\frac{\Q(S_{m+1}=x\pm 1,\ S_m=x)}{\Q(S_m=x)}\\
&\qquad =\frac{\P[\frac{h(m+1,S_{m+1})}{h(0,0)}1_{\{T_0>m+1\}}1_{\{S_{m+1}=x\pm 1,\ S_m=x\}}]}{\P[\frac{h(m,S_m)}{h(0,0)}1_{\{T_0>m\}}1_{\{S_m=x\}}]}\\
&\qquad =\frac{a_{m+1}|x\pm 1|}{a_m|x|}\cdot \frac{\P(S_{m+1}\neq 0,\ T_0>m,\ S_{m+1}=x\pm 1,\ S_m=x)}{\P(T_0>m,\ S_m=x)}\\
&\qquad =\frac{a_{m+1}|x\pm 1|}{a_m|x|}\cdot \frac{\P(T_0>m,\ S_m=x)\P(S_{m+1}=x\pm 1\mid S_m=x)}{\P(T_0>m,\ S_m=x)}\\
&\qquad =\frac{a_{m+1}|x\pm 1|}{a_m|x|}\P(S_{m+1}=x\pm 1\mid S_m=x).
\stepcounter{equation}\tag{\theequation}
\end{align*}
By (\ref{sequence-an}), we have
\begin{align*}
\label{Thm1.3-pfeq2}
\frac{a_{m+1}}{a_m}&=\frac{\Gamma(m+1)}{\Gamma(m+2p)}\cdot \frac{\Gamma(m+2p-1)}{\Gamma(m)}\\
&=\frac{m\Gamma(m)}{(m+2p-1)\Gamma(m+2p-1)}\cdot \frac{\Gamma(m+2p-1)}{\Gamma(m)}\\
&=\frac{m}{m+2p-1}.
\stepcounter{equation}\tag{\theequation}
\end{align*}
Therefore, combining (\ref{Thm1.3-pfeq1}), (\ref{Thm1.3-pfeq2}), and (\ref{ERW-transitionprobability}), we obtain (\ref{mainthm2-eq3}). This completes the proof.
\end{proof}

\subsection{The Critical Case}
In this subsection, we consider the critical case $p=\frac{3}{4}$.

The argument in this case is almost the same as that for the diffusive case $0<p<\frac{3}{4}$. The parts of Bertoin's embedding argument that do not rely on asymptotic estimates remain valid for the critical case $p=\frac{3}{4}$. Therefore, Lemma \ref{Lem2.1} also holds when $p=\frac{3}{4}$.

In view of (\ref{An-asymp}), we note that the order of $A_n$ in the critical case $p=\frac{3}{4}$ differs from that in the diffusive case $0<p<\frac{3}{4}$. Consequently, an analogue of Lemma \ref{Lem2.2} also holds in the critical case.

\begin{lem}
\label{Lem2.5}
Let $y\in \Z\setminus \{0\}$. Assume that $\P(S_m=y)>0$. Then, for $\e>0$ and $r\ge 1$, there exists a constant $c_{\e,r,m,y}>0$ such that
\begin{align}
\label{Lem2.5-main}
\mathbf{P}_{ya_m}\Big(|\tau^{(N)}-(A_{m+N}-A_m)|>\e\log (m+N)\Big)\le c_{\e,r,m,y}\left(\frac{\log(m+N)}{m+N}\right)^{r}.
\end{align}
\end{lem}
\begin{proof}
This can be proved in exactly the same way as Lemma \ref{Lem2.2}. The only difference is that, whereas \cite[Proposition 4.2]{Bertoin-zero} was used in Lemma \ref{Lem2.2}, \cite[Proposition 4.2]{Fang} is used here instead.
\end{proof}

By applying this lemma and observing that $\log N (\frac{\log N}{N})^r\to 0$ as $N\to \infty$ holds for every $r>0$, we see that Theorem \ref{Prop2.4} remains valid when $p=\frac{3}{4}$. Therefore, by arguing exactly as in the diffusive case $0<p<\frac{3}{4}$, we can also prove Theorems \ref{mainthm-1} and \ref{mainthm-2} when $p=\frac{3}{4}$.


\section{Embedding into the Three-Dimensional Bessel Process}
\label{S3}
In this section, we embed the ERW conditioned to avoid zero, constructed in the previous section, into a three-dimensional Bessel process. 

We retain the notation $(R_t)_{t\ge0}$ and $(\mathscr{P}_x)_{x\ge0}$ introduced above. Recall that the sequence of stopping times $(\sigma_n)$ of $(R_t)$ is defined by
\begin{align}\label{sigma}
\begin{aligned}
\sigma_0&:=0,\\
\sigma_1&:=\inf\{t\ge 0:\ R_t=a_1\},\\
\sigma_{n+1}&:=\inf \left\{t>\sigma_n:\ R_t\notin  \left(a_{n+1}\Big(\frac{R_{\sigma_n}}{a_n}-1\Big),a_{n+1}\Big(\frac{R_{\sigma_n}}{a_n}+1\Big)\right)\right\}\qquad \text{for}\ n\ge 1.
\end{aligned}
\end{align}

Now let us prove Proposition \ref{Prop1.6}.

\begin{proof}[Proof of Proposition \ref{Prop1.6}] Let $Y_n:=a_n |S_n|$. By (\ref{mainthm2-eq1}), we have
\begin{align*}
\mathscr{P}_0(R_{\sigma_0}=0)=\mathscr{P}_0(R_0=0)=1=\Q(Y_0=0).
\end{align*}
By continuity of $(R_t)$ and (\ref{mainthm2-eq2}), we have
\begin{align*}
\mathscr{P}_0(R_{\sigma_1}=a_1)=1=\Q(Y_1=a_1).
\end{align*}
Next, by continuity of $(R_t)$ and (\ref{CERW-when1}), we have
\begin{align*}
\mathscr{P}_0\Big(R_{\sigma_{n+1}}=2a_{n+1}\mid R_{\sigma_n}=a_n\Big)=1=\Q\Big(Y_{n+1}=2a_{n+1}\mid Y_n=a_n\Big)
\end{align*}
for $n\ge 1$. Finally, using the fact that the scale function of $(R_t)$ is given by $x\mapsto -\frac{1}{x}$ (see, e.g., \cite[page 137]{Borodin-Salminen}), together with (\ref{mainthm2-eq3}), we have
\begin{align*}
\mathscr{P}_0\Big(R_{\sigma_{n+1}}=a_{n+1}(x+1)\mid R_{\sigma_n}=a_n x\Big)&=\frac{\frac{1}{a_{n+1}(x-1)}-\frac{1}{a_n x}}{\frac{1}{a_{n+1}(x-1)}-\frac{1}{a_{n+1}(x+1)}}\\
&=\frac{a_{n+1}}{a_n}\cdot \frac{x+1}{x}\cdot \left\{\frac{1}{2}+\left(p-\frac{1}{2}\right)\frac{x}{n}\right\}\\
&=\Q\Big(Y_{n+1}=a_{n+1}(x+1)\mid Y_n=a_n x\Big)
\stepcounter{equation}\tag{\theequation}
\end{align*}
for $n\ge 1$ and $x\ge 2$. It follows from the above that the processes $(R_{\sigma_{n}})_{n=0,1,2,...}$ under $\mathscr{P}_0$ and $(Y_n)_{n=0,1,2,...}$ under $\Q$ are Markov chains with the same transition probabilities and the same initial distribution. This completes the proof.
\end{proof}

Although the asymptotic behavior of the general $r$-th moment will be
investigated later, the second-moment asymptotics are needed in the
proof of Proposition \ref{asym2}, so we state them here. Since the
proof consists only of elementary calculations involving sequences,
it is deferred to the Appendix.

\begin{lem}\label{asym1}
The second moment of the C-ERW satisfies
\begin{align*}
\mathbb Q[S_n^2]
\sim
\begin{cases}
\displaystyle
\frac{3}{3-4p}n,
& p<\frac34,\\[2mm]
\displaystyle
3n\log n,
& p=\frac34.
\end{cases}
\end{align*}
\end{lem}

We use $C(E,F)$ to denote the space of continuous functions from $E$ to $F$.

For later use, we realize the BES(3) introduced above as the radial part of a three-dimensional Brownian motion. Let $x\ge0$, and let
$((\mathbb B_t)_{t\ge0},\mathbf P_{(x,0,0)})$ be a
three-dimensional Brownian motion starting from $(x,0,0)$. Define
\begin{align*}
\Phi:
C([0,\infty),\mathbb R^3)
\to
C([0,\infty),[0,\infty))
\end{align*}
by
\begin{align*}
\Phi(\omega)(t):=|\omega(t)|\qquad \text{for}\ t\ge 0
\end{align*}
where $|\cdot|$ denotes the Euclidean norm. Set
\begin{align*}
\mathscr P_x
:=
\mathbf P_{(x,0,0)}\circ\Phi^{-1}.
\end{align*}
Moreover, let
\begin{align*}
R_t:
C([0,\infty),[0,\infty))
\to
[0,\infty)
\end{align*}
be the coordinate map defined by
\begin{align*}
R_t(\omega):=\omega(t)
\qquad \text{for}\ t\ge 0.
\end{align*}
Since the radial part of a three-dimensional Brownian motion is a
three-dimensional Bessel process
(see, e.g., \cite[Section 3.3.C]{Karatzas-Shreve}),
$((R_t)_{t\ge0},\mathscr P_x)$ is a three-dimensional Bessel process
starting from $x$.

We set
\begin{align*}
\mathcal{G}_t:=\sigma(R_s:s\le t).
\end{align*}
We also define
\begin{align*}
\varsigma_{n+1}:=\sigma_{n+1}-\sigma_n\qquad \text{for}\ n\ge 0,
\end{align*}
where $\sigma_n$ is defined in \eqref{sigma}.

For later use, set $Z_0:=0$ and define
\begin{align*}
Z_n:=\frac{R_{\sigma_n}}{a_n}\qquad \text{for}\ n\ge 1.
\end{align*}
In particular, $0\le Z_n\le n$, $\mathscr{P}_0$-a.s. By Proposition \ref{Prop1.6}, the process $(Z_n)_{n=0,1,2,\ldots}$ under $\mathscr{P}_0$ has the same law as $(|S_n|)_{n=0,1,2,\ldots}$ under $\Q$.

The following lemma will be used in the proof of Proposition \ref{asym2}.
\begin{lem}\label{Lem1}
For any $s>0$, there exists a constant $c_s>0$ such that, for every $n\ge0$,
\begin{align}\label{ineq5}
\mathscr{P}_0[\varsigma_{n+1}^{s}\mid\mathcal{G}_{\sigma_n}]
\le c_sa_{n+1}^{2s},
\quad\text{$\mathscr{P}_0$-a.s.}
\end{align}
\end{lem}

\begin{proof}
For $m\ge0$, let $I_m:=(m-1,m+1)$ and
\begin{align*}
\widetilde{T}_m=\inf\{t>0: R_t\notin I_m\}.
\end{align*}
We first show that there exists $\rho\in(0,1)$ such that, for every $m\ge0$ and $x\in I_m\cap[0,\infty)$,
\begin{align}\label{ineq3}
\mathscr P_x\bigl(\widetilde{T}_m>1\bigr)\le 1-\rho.
\end{align}

Note that
\begin{align*}
(R_t)_{t\ge0}\ \text{under }\mathscr P_x
\overset{d}{=}
(|\mathbb B_t|)_{t\ge0}\ \text{under }\mathbf P_{(x,0,0)}.
\end{align*}
Let $\mathbf e_1=(1,0,0)$ and set
\begin{align*}
V_t:=\mathbf e_1\cdot\mathbb B_t-x.
\end{align*}
Then $(V_t)_{t\ge0}$ is a standard one-dimensional Brownian motion.
Since $x\in I_m$, on the event $\{V_1\ge2\}$, we have
\begin{align*}
|\mathbb B_1|
\ge \mathbf e_1\cdot\mathbb B_1
=x+V_1
\ge x+2
>m+1.
\end{align*}
Thus, we have
\begin{align*}
\mathbf P_{(x,0,0)}(V_1\ge2)
&\le
\mathbf P_{(x,0,0)}(|\mathbb B_1|>m+1)\\
&=
\mathscr P_x(R_1>m+1)\\
&\le
\mathscr P_x\bigl(\widetilde T_m\le1\bigr),
\end{align*}
where the last inequality follows from the continuity of $R$.
Let
\begin{align*}
\rho:=\mathbf P_0(B_1\ge2).
\end{align*}
Then $\rho\in(0,1)$ and
\begin{align*}
\mathscr P_x\bigl(\widetilde T_m>1\bigr)
&\le
1-\mathbf P_{(x,0,0)}(V_1\ge2)\\
&=1-\rho.
\end{align*}
Thus, \eqref{ineq3} follows.

Next, define
\begin{align*}
\widetilde{R}_t:=\frac{R_{\sigma_n+a_{n+1}^2t}}{a_{n+1}}\qquad \text{for}\ t\ge 0.
\end{align*}
Then
\begin{align*}
\widetilde{\varsigma}_{n+1}
:=\frac{\varsigma_{n+1}}{a_{n+1}^2}
=\inf\{t>0: \widetilde{R}_t\notin I_{Z_n}\}.
\end{align*}
Fix $n\ge0$. We next show that, for every integer $k\ge0$,
\begin{align}\label{ineq4}
\mathscr P_0(\widetilde{\varsigma}_{n+1}>k\mid\mathcal{G}_{\sigma_n})
\le(1-\rho)^k.
\end{align}
Let
\begin{align*}
\widetilde{\varsigma}^{(k)}
:=\inf\{t>0: \widetilde{R}_{k+t}\notin I_{Z_n}\}\qquad \text{for}\ k\ge 0.
\end{align*}
Since $\widetilde{\varsigma}_{n+1}$ is a stopping time with respect to
$(\mathcal G_{\sigma_n+a_{n+1}^2t})_{t\ge0}$ and
\begin{align*}
\{\widetilde{\varsigma}_{n+1}>k+1\}
=
\{\widetilde{\varsigma}_{n+1}>k\}
\cap
\{\widetilde{\varsigma}^{(k)}>1\},
\end{align*}
we have
\begin{align}\label{equal3}
\mathscr P_0(\widetilde{\varsigma}_{n+1}>k+1
\mid\mathcal{G}_{\sigma_n+a_{n+1}^2k})
=
1_{\{\widetilde{\varsigma}_{n+1}>k\}}
\mathscr P_0(\widetilde{\varsigma}^{(k)}>1
\mid\mathcal{G}_{\sigma_n+a_{n+1}^2k}).
\end{align}
Here,
\begin{align*}
\{\widetilde{\varsigma}^{(k)}>1\}
=
\bigcup_{m=0}^n
\{Z_n=m\}
\cap
\{\widetilde{R}_{k+t}\in I_m
\text{ for all }0<t\le1\}.
\end{align*}
By the strong Markov property and the scaling property, conditioned on
$\mathcal{G}_{\sigma_n+a_{n+1}^2k}$,
$(\widetilde{R}_{k+t})_{t\ge0}$ is a three-dimensional Bessel process started from $\tilde{R}_k$. Hence,
\begin{align*}
\mathscr P_0(\widetilde{\varsigma}^{(k)}>1
\mid\mathcal{G}_{\sigma_n+a_{n+1}^2k})
=
\sum_{m=0}^n
\mathscr P_{\widetilde{R}_k}\bigl(\widetilde{T}_{m}>1\bigr)
1_{\{Z_n=m\}}.
\end{align*}
On $\{\widetilde{\varsigma}_{n+1}>k\}\cap\{Z_n=m\}$, we have $\widetilde{R}_k\in I_m$. Hence, by \eqref{ineq3} and \eqref{equal3}, we have
\begin{align*}
\mathscr P_0(\widetilde{\varsigma}_{n+1}>k+1
\mid\mathcal{G}_{\sigma_n+a_{n+1}^2k})
\le
1_{\{\widetilde{\varsigma}_{n+1}>k\}}(1-\rho).
\end{align*}
By the tower property, we have
\begin{align*}
\mathscr P_0(\widetilde{\varsigma}_{n+1}>k+1
\mid\mathcal{G}_{\sigma_n})
\le
\mathscr P_0(\widetilde{\varsigma}_{n+1}>k
\mid\mathcal{G}_{\sigma_n})(1-\rho).
\end{align*}
Hence, inductively, we have
\begin{align*}
\mathscr P_0(\widetilde{\varsigma}_{n+1}>k
\mid\mathcal{G}_{\sigma_n})
\le (1-\rho)^{k}.
\end{align*}
Thus, \eqref{ineq4} has been proved.

For $t\in[k,k+1)$, it follows from \eqref{ineq4} that
\begin{align*}
\mathscr P_0(\widetilde{\varsigma}_{n+1}>t
\mid\mathcal{G}_{\sigma_n})
\le
\mathscr P_0(\widetilde{\varsigma}_{n+1}>k
\mid\mathcal{G}_{\sigma_n})
\le(1-\rho)^k.
\end{align*}
Hence,
\begin{align*}
\mathscr{P}_0[\widetilde{\varsigma}_{n+1}^s
\mid\mathcal{G}_{\sigma_n}]
&=
s\int_0^\infty
t^{s-1}
\mathscr P_0(\widetilde{\varsigma}_{n+1}>t
\mid\mathcal{G}_{\sigma_n})dt\\
&=
s\sum_{k=0}^\infty
\int_{k}^{k+1}
t^{s-1}
\mathscr P_0(\widetilde{\varsigma}_{n+1}>t
\mid\mathcal{G}_{\sigma_n})dt\\
&\le
\sum_{k=0}^\infty
((k+1)^s-k^s)(1-\rho)^k\\
&\le
\sum_{k=0}^\infty
(k+1)^s(1-\rho)^k
<\infty.
\end{align*}
Let
\begin{align*}
c_s=\sum_{k=0}^\infty (k+1)^s(1-\rho)^k.
\end{align*}
Then, since
$\varsigma_{n+1}=a_{n+1}^2\widetilde{\varsigma}_{n+1}$,
\begin{align*}
\mathscr{P}_0[\varsigma_{n+1}^{s}\mid\mathcal{G}_{\sigma_n}]
\le c_sa_{n+1}^{2s}.
\end{align*}
Thus, \eqref{ineq5} has been proved.
\end{proof}

Now let us prove Proposition \ref{asym2}.
\begin{proof}[Proof of Proposition \ref{asym2}]
First, we show that
\begin{align}\label{asconv2}
\lim_{n\to\infty}\frac{\sigma_n}{A_n}=1 \quad\text{$\mathscr{P}_0$-a.s.}
\end{align}
Since $(R_t^2-R_0^2-3t)_{t\ge0}$ is a martingale and $R$ has the strong
Markov property, the process
\begin{align*}
\left(
R_{\sigma_n+t}^2-R_{\sigma_n}^2-3t
\right)_{t\ge0}
\end{align*}
is a martingale with respect to
$\left(\mathcal G_{\sigma_n+t}\right)_{t\ge0}$. Since $\varsigma_{n+1}^{(N)}=\varsigma_{n+1}\wedge N$ is a bounded stopping time with respect to this filtration, the optional stopping theorem yields
\begin{align*}
\mathscr{P}_0\left[R_{\sigma_n+\varsigma_{n+1}^{(N)}}^2-R_{\sigma_n}^2\mid\mathcal{G}_{\sigma_n}\right]=3\mathscr{P}_0[\varsigma_{n+1}^{(N)}\mid\mathcal{G}_{\sigma_n}].
\end{align*}
By the definition of $\sigma_n$, we have $0\le R_{\sigma_n+\varsigma_{n+1}^{(N)}}^2\le a_{n+1}^2(n+1)^2$. Therefore, letting $N\to\infty$ in the above identity, the dominated convergence theorem and the monotone convergence theorem yield
\begin{align}\label{equal4}
\mathscr{P}_0[R_{\sigma_{n+1}}^2-R_{\sigma_n}^2\mid\mathcal{G}_{\sigma_n}]=3\mathscr{P}_0[\varsigma_{n+1}\mid\mathcal{G}_{\sigma_n}].
\end{align}
By the strong Markov property and the transition probabilities
computed in the proof of Proposition \ref{Prop1.6}, a calculation
similar to that in the proof of Lemma \ref{asym1} yields
\begin{align}\label{equal5}
\mathscr P_0\left[
R_{\sigma_{n+1}}^2\mid
\mathcal G_{\sigma_n}
\right]
=
a_{n+1}^2
\left(
\frac{n+6p-3}{n+2p-1}Z_n^2
+
\frac{3n+2p-1}{n+2p-1}
\right).
\end{align}By \eqref{equal4} and \eqref{equal5}, we have
\begin{align*}
\mathscr{P}_0[\varsigma_{n+1}\mid\mathcal{G}_{\sigma_n}]=a_{n+1}^2\frac{3n+2p-1}{3(n+2p-1)}\left(1-\left(\frac{(2p-1)Z_n}{n}\right)^2\right).
\end{align*}
Hence, there exists a constant $c_p>0$ such that
\begin{align}\label{ineq6}
|\mathscr{P}_0[\varsigma_{n+1}\mid\mathcal{G}_{\sigma_n}]-a_{n+1}^2|\le c_pa_{n+1}^2\left(\frac{1}{n}+\frac{Z_n^2}{n^2}\right).
\end{align}
We now consider the decomposition
\begin{align}
\sigma_n=C_n+D_n.
\end{align}
Here,
\begin{align*}
C_n=\sigma_1+\sum_{j=1}^{n-1}\mathscr{P}_0[\varsigma_{j+1}\mid\mathcal{G}_{\sigma_j}],\qquad D_n=\sum_{j=1}^{n-1}\xi_{j+1},
\end{align*}
where $\xi_{j+1}:=\varsigma_{j+1}-\mathscr{P}_0[\varsigma_{j+1}\mid\mathcal{G}_{\sigma_j}]$. We show that
\begin{align*}
\lim_{n\to\infty}\frac{C_n}{A_n}=1,\quad\lim_{n\to\infty}\frac{D_n}{A_n}=0\qquad\text{$\mathscr P_0$-a.s.}
\end{align*}
By \eqref{an-asymp} and \eqref{An-asymp},
\begin{align*}
\frac{a_{n+1}^2}{A_n}=\begin{cases}
O(n^{-1})&(p<\frac34)\\
O((n\log n)^{-1})&(p=\frac34)
\end{cases}
\end{align*}
and hence
\begin{align*}
\sum_{j=2}^\infty\frac{a_{j+1}^2}{jA_j}
\end{align*}
converges. Thus, by Kronecker's lemma, we have
\begin{align}\label{conv1}
\lim_{n\to\infty}\frac{1}{A_n}\sum_{j=1}^{n-1}\frac{a_{j+1}^2}{j}=0.
\end{align}
Next, by Lemma \ref{asym1}, we have
\begin{align*}
\mathscr{P}_0\left[\frac{a_{n+1}^2Z_n^2}{n^2A_n}\right]=\Q\left[\frac{a_{n+1}^2S_n^2}{n^2A_n}\right]=O(n^{-2}).
\end{align*}
Hence, by Fubini's theorem, we have
\begin{align*}
\mathscr{P}_0\left[\sum_{n=2}^\infty\frac{a_{n+1}^2Z_n^2}{n^2A_n}\right]=\sum_{n=2}^\infty \mathscr{P}_0\left[\frac{a_{n+1}^2Z_n^2}{n^2A_n}\right]<\infty.
\end{align*}
Thus, the series 
\begin{align*}
\sum_{n=2}^\infty\frac{a_{n+1}^2Z_n^2}{n^2A_n}
\end{align*}
converges almost surely. Thus, by Kronecker's lemma, we have
\begin{align}\label{conv2}
\lim_{n\to\infty}\frac{1}{A_n}\sum_{j=1}^{n-1}\frac{a_{j+1}^2Z_j^2}{j^2}=0\quad\text{$\mathscr{P}_0$-a.s.}
\end{align}
By \eqref{ineq6}, \eqref{conv1}, \eqref{conv2}, and \eqref{An-asymp}, we obtain
\begin{align}\label{asconv}
\begin{aligned}
\left|\frac{C_n}{A_n}-1\right|&\le \frac{|\sigma_1-a_1^2|}{A_n}+\frac{1}{A_n}\left|\sum_{j=1}^{n-1}\left(\mathscr{P}_0[\varsigma_{j+1}\mid\mathcal{G}_{\sigma_j}]-a_{j+1}^2\right)\right|\\
&\le\frac{|\sigma_1-a_1^2|}{A_n}+\frac{c_p}{A_n}\sum_{j=1}^{n-1}a_{j+1}^2\left(\frac{1}{j}+\frac{Z_j^2}{j^2}\right)\xrightarrow[n\to\infty]{}0\quad\text{$\mathscr{P}_0$-a.s.}
\end{aligned}
\end{align}
By Lemma \ref{Lem1}, we have
\begin{align*}
\mathscr{P}_0[\xi_{n+1}^2\mid\mathcal{G}_{\sigma_n}]=\mathscr{P}_0[\varsigma_{n+1}^2\mid\mathcal{G}_{\sigma_n}]-(\mathscr{P}_0[\varsigma_{n+1}\mid\mathcal{G}_{\sigma_n}])^2\le c_2a_{n+1}^4.
\end{align*}
By \eqref{an-asymp} and \eqref{An-asymp},
\begin{align*}
\frac{a_{n+1}^4}{A_n^2}=\begin{cases}
O(n^{-2})&(p<\frac34)\\
O((n\log n)^{-2})&(p=\frac34)
\end{cases}
\end{align*}
and hence
\begin{align*}
\sum_{n=1}^\infty \mathscr{P}_0\left[\frac{\xi_{n+1}^2}{A_n^2}\mid\mathcal{G}_{\sigma_n}\right]\le c_2\sum_{n=1}^\infty \frac{a_{n+1}^4}{A_n^2}<\infty.
\end{align*}
It follows from \cite[Theorem 2.15]{Hall-Heyde} that
\begin{align*}
\sum_{n=1}^\infty\frac{\xi_{n+1}}{A_n}
\end{align*}
converges almost surely. Thus, by Kronecker's lemma, we have
\begin{align*}
\lim_{n\to\infty}\frac{D_n}{A_n}=0\quad\text{$\mathscr{P}_0$-a.s.}
\end{align*}
This proves \eqref{asconv2}.

Next, we show that
\begin{align}\label{Lrconv}
\lim_{n\to\infty}\frac{\sigma_n}{A_n}=1 \quad\text{in $L^r(\mathscr{P}_0)$}.
\end{align}
We show that
\begin{align*}
\lim_{n\to\infty}\frac{C_n}{A_n}=1,\quad\lim_{n\to\infty}\frac{D_n}{A_n}=0\quad\text{in $L^r(\mathscr{P}_0)$}.
\end{align*}
By \eqref{ineq6}, we have
\begin{align*}
|\mathscr{P}_0[\varsigma_{n+1}\mid\mathcal{G}_{\sigma_n}]-a_{n+1}^2|\le c_pa_{n+1}^2.
\end{align*}
Thus, we have
\begin{align*}
\left|\frac{C_n}{A_n}-1\right|^r&\le\left(\frac{|\sigma_1-a_1^2|}{A_n}+\left|\frac{1}{A_n}\sum_{j=1}^{n-1}\left(\mathscr{P}_0[\varsigma_{j+1}\mid\mathcal{G}_{\sigma_j}]-a_{j+1}^2\right)\right|\right)^r\\
&\le c_r\left(\frac{|\sigma_1-a_1^2|^r}{A_1^r}+c_p^r\right)\\
&\le c_{r,p}(1+\sigma_1^r).
\end{align*}
By Lemma \ref{Lem1}, $\sigma_1^r\in L^1(\mathscr{P}_0)$. Therefore, \eqref{asconv} and the dominated convergence theorem imply that
\begin{align*}
\lim_{n\to\infty}\frac{C_n}{A_n}=1\quad\text{in $L^r(\mathscr{P}_0)$}.
\end{align*}
By Jensen's inequality and Lemma \ref{Lem1}, for $r\ge1$,
\begin{align}
\label{ineq2}
\begin{aligned}
\mathscr{P}_0[|\xi_{n+1}|^r\mid\mathcal{G}_{\sigma_n}]&\le 2^{r-1}\left(\mathscr{P}_0[\varsigma_{n+1}^r\mid\mathcal{G}_{\sigma_n}]+(\mathscr{P}_0[\varsigma_{n+1}\mid\mathcal{G}_{\sigma_n}])^r\right)\\
&\le c_r a_{n+1}^{2r}.
\end{aligned}
\end{align}
By \cite[Theorem 2.12]{Hall-Heyde}, for every $r\ge2$, there exists a constant $c_r>0$ such that
\begin{align*}
\mathscr{P}_0[|D_n|^r]\le c_r\left(\mathscr{P}_0\left[\left(\sum_{j=1}^{n-1}\mathscr{P}_0[\xi_{j+1}^2\mid\mathcal{G}_{\sigma_j}]\right)^{r/2}\right]+\sum_{j=1}^{n-1}\mathscr{P}_0[|\xi_{j+1}|^r]\right).
\end{align*}
The above inequality and \eqref{ineq2} yield
\begin{align*}
\mathscr{P}_0\left[\left|\frac{D_n}{A_n}\right|^r\right]\le c_r\left(\left(\frac{1}{A_n^2}\sum_{j=1}^{n-1}a_{j+1}^4\right)^{r/2}+\frac{1}{A_n^r}\sum_{j=1}^{n-1}a_{j+1}^{2r}\right).
\end{align*}
By \eqref{an-asymp} and \eqref{An-asymp}, for every $r\ge2$,
\begin{align*}
\lim_{n\to\infty}\frac{1}{A_n^r}\sum_{j=1}^{n-1}a_{j+1}^{2r}=0.
\end{align*}
Thus, for $r\ge2$,
\begin{align*}
\lim_{n\to\infty}\mathscr{P}_0\left[\left|\frac{D_n}{A_n}\right|^r\right]=0.
\end{align*}
When $0<r<2$, Jensen's inequality yields
\begin{align*}
\mathscr{P}_0\left[\left|\frac{D_n}{A_n}\right|^r\right]\le\mathscr{P}_0\left[\left|\frac{D_n}{A_n}\right|^2\right]^{r/2}\xrightarrow[n\to\infty]{}0.
\end{align*}
Therefore, for every $r>0$,
\begin{align*}
\lim_{n\to\infty}\frac{D_n}{A_n}=0\quad\text{in $L^r(\mathscr{P}_0)$}.
\end{align*}
This proves \eqref{Lrconv}.
\end{proof}


\section{Properties of the ERW Conditioned to Avoid Zero}
\label{S4}
Now let us prove the transience assertion in Theorem \ref{mainthm-properties-ERW}.
\begin{proof}[Proof of the transience assertion in Theorem \ref{mainthm-properties-ERW}]
We prove that
\begin{align*}
\Q(|S_n|\to\infty)=1.
\end{align*}
We note that, by taking $d=3$ and $f(t)=t^\alpha$ in the
Dvoretzky--Erd\H{o}s test (see, e.g., \cite[Theorem 3.22]{Morters-Peres}),
\begin{align*}
\lim_{t\to\infty}\frac{R_t}{t^\alpha}=\infty
\quad\text{$\mathscr P_0$-a.s.}
\end{align*}
for every $\alpha\in(0,\frac12)$.
By Proposition \ref{asym2}, $\sigma_n\to\infty$ almost surely, and hence
\begin{align*}
\lim_{n\to\infty}\frac{R_{\sigma_n}}{\sigma_n^\alpha}=\infty\quad\text{$\mathscr P_0$-a.s.}
\end{align*}

First, consider the case $p<\frac34$. Choose $\beta$ such that
\begin{align*}
\max\left\{0,\frac{1-2p}{3-4p}\right\}<\beta<\frac12.
\end{align*}
By Proposition \ref{asym2} and \eqref{an-asymp}, we obtain
\begin{align*}
\lim_{n\to\infty}Z_n
=\lim_{n\to\infty}\frac{R_{\sigma_n}}{\sigma_n^\beta}\cdot
\frac{\sigma_n^\beta}{a_n}=\infty\quad\text{$\mathscr P_0$-a.s.}
\end{align*}

Next, in the case $p=\frac34$, by \eqref{an-asymp}, we obtain
\begin{align*}
\lim_{n\to\infty}Z_n
=\lim_{n\to\infty}\frac{R_{\sigma_n}}{\sigma_n^\alpha}\cdot
\frac{\sigma_n^\alpha}{a_n}=\infty\quad\text{$\mathscr P_0$-a.s.}
\end{align*}
Therefore, by Proposition \ref{Prop1.6}, we have
\begin{align*}
\mathscr{P}_0(Z_n\to\infty)
=\Q(|S_n|\to\infty)
=1.
\end{align*}
The proof is complete.
\end{proof}
We next prove the scaling limit assertions in Theorem \ref{mainthm-properties-ERW}.

For $T,N>0$, let $D_0([0,T],[0,N])$ denote the space of all
nondecreasing c\`adl\`ag functions from $[0,T]$ into $[0,N]$, endowed
with the relative Skorokhod topology inherited from
$D([0,T])$. We equip product spaces with the product topology.
\begin{proof}[Proof of the scaling limit assertions in Theorem
\ref{mainthm-properties-ERW}]
By Proposition \ref{Prop1.6}, it suffices
to prove the corresponding convergences for $(Z_n)$ under
$\mathscr P_0$.

First, suppose that $p<\frac34$. We have
\begin{align}\label{equal}
\frac{Z_{[nt]}}{\sqrt n}
=
\frac{R_{\sigma_{[nt]}}}{a_{[nt]}\sqrt n}
=
\frac{\sqrt{A_n}}{a_{[nt]}\sqrt n}\cdot
\frac{R_{\sigma_{[nt]}}}{\sqrt{A_n}}.
\end{align}
Define
\begin{align*}
R^{(n)}(t):=\frac{R_{A_nt}}{\sqrt{A_n}},
\qquad
\Lambda^{(n)}(t):=\frac{\sigma_{[nt]}}{A_n},
\qquad
\phi(t):=t^{3-4p}.
\end{align*}
By the scaling property of the Bessel process,
$R^{(n)}\overset{d}{=}R$ as random elements of $D([0,\infty))$.

By Proposition \ref{asym2}, the regular variation of $(A_n)$, the uniform convergence theorem (see, e.g., \cite[Theorem 1.5.2]{Bingham}), and
monotonicity, we have
\begin{align}\label{uc}
\lim_{n\to\infty}\sup_{t\in[0,T]}|\Lambda^{(n)}(t)-\phi(t)|=0\qquad\text{$\mathscr P_0$-a.s.}
\end{align}
Take $N>\phi(T)$ and define
\begin{align*}
\Lambda_N^{(n)}(t):=\Lambda^{(n)}(t)\wedge N.
\end{align*}
By \eqref{uc} and $\phi([0,T])\subset[0,N]$, we have
\begin{align*}
\lim_{n\to\infty}\sup_{t\in[0,T]}|\Lambda^{(n)}_N(t)-\phi(t)|=0\qquad\text{$\mathscr P_0$-a.s.}
\end{align*}
Hence, by
Slutsky's theorem, as $n\to\infty$,
\begin{align*}
&\left(
\left(R^{(n)}(t)\right)_{0\le t\le N},
\left(\Lambda_N^{(n)}(t)\right)_{0\le t\le T}
\right)
\Longrightarrow
\left(
(R_t)_{0\le t\le N},
(\phi(t))_{0\le t\le T}
\right)
\end{align*}
in $D([0,N])\times D_0([0,T],[0,N])$.

We next apply the composition map
\begin{align*}
\Psi:
D([0,N])\times D_0([0,T],[0,N])
\to D([0,T]),
\qquad
\Psi(f,g):=f\circ g.
\end{align*}
The map $\Psi$ is Borel measurable (see, e.g., \cite[p.~232]{Billingsley}). Moreover, $\Psi$ is continuous on
$C([0,N])\times D_0([0,T],[0,N])$ (see, e.g.,
\cite[Theorem 3.1]{Whitt}).

Since
\begin{align*}
\left(
(R_t)_{0\le t\le N},
(\phi(t))_{0\le t\le T}
\right)
\in
C([0,N])\times D_0([0,T],[0,N]),
\end{align*}
the continuous mapping theorem (see, e.g., \cite[Theorem 5.1]{Billingsley}) yields, as $n\to\infty$,
\begin{align*}
\left(
R^{(n)}\bigl(\Lambda_N^{(n)}(t)\bigr)
\right)_{0\le t\le T}
\implies
\left(
R_{\phi(t)}
\right)_{0\le t\le T}
\quad\text{in }D([0,T]).
\end{align*}

Furthermore, since $\Lambda^{(n)}\to\phi$ uniformly on $[0,T]$ and
$N>\phi(T)$,
\begin{align*}
\lim_{n\to\infty}\mathscr P_0\left(
\Lambda^{(n)}([0,T])\subset[0,N]
\right)=1.
\end{align*}
On this event,
$\Lambda_N^{(n)}=\Lambda^{(n)}$ on $[0,T]$. Therefore, the two
composed processes coincide with probability tending to one, and hence
\begin{align*}
\left(
R^{(n)}\bigl(\Lambda^{(n)}(t)\bigr)
\right)_{0\le t\le T}
\implies
\left(
R_{\phi(t)}
\right)_{0\le t\le T}
\quad\text{in }D([0,T]).
\end{align*}
Equivalently,
\begin{align}\label{time-change-subcritical}
\left(
\frac{R_{\sigma_{[nt]}}}{\sqrt{A_n}}
\right)_{0\le t\le T}
\implies
\left(
R_{t^{3-4p}}
\right)_{0\le t\le T}
\quad\text{in }D([0,T]).
\end{align}

Let $0<\delta<T$. To handle the deterministic factor in \eqref{equal}, note that, by
\eqref{an-asymp}, $a_n$ is regularly varying with index
$1-2p$. Hence, the uniform convergence theorem
(see, e.g., \cite[Theorem 1.5.2]{Bingham}) gives
\begin{align*}
\lim_{n\to\infty}\sup_{t\in[\delta,T]}
\left|
\frac{a_n}{a_{[nt]}}-t^{2p-1}
\right|=0.
\end{align*}
Therefore, by
\eqref{An-asymp},
\begin{align*}
&\sup_{t\in[\delta,T]}
\left|
\frac{\sqrt{A_n}}{a_{[nt]}\sqrt n}
-
\frac{1}{\sqrt{3-4p}}t^{2p-1}
\right|
\\
&\le
\left|
\frac{\sqrt{A_n}}{a_n\sqrt n}
-
\frac{1}{\sqrt{3-4p}}
\right|
\sup_{t\in[\delta,T]}\frac{a_n}{a_{[nt]}}
+
\frac{1}{\sqrt{3-4p}}
\sup_{t\in[\delta,T]}
\left|
\frac{a_n}{a_{[nt]}}-t^{2p-1}
\right|
\xrightarrow[n\to\infty]{}0.
\end{align*}
Thus, by \eqref{equal} and Slutsky's theorem, we have
\begin{align}\label{scaling-Z-subcritical}
\left(
\frac{Z_{[nt]}}{\sqrt n}
\right)_{\delta\le t\le T}
\implies
\left(
\frac{1}{\sqrt{3-4p}}
t^{2p-1}R_{t^{3-4p}}
\right)_{\delta\le t\le T}
\quad\text{in }D([\delta,T]).
\end{align}
Since $\delta>0$ and $T>0$ are arbitrary, we have
\begin{align*}
\left(
\frac{Z_{[nt]}}{\sqrt n}
\right)_{t>0}
\implies
\left(
\frac{1}{\sqrt{3-4p}}
t^{2p-1}R_{t^{3-4p}}
\right)_{t>0}
\quad\text{in }D((0,\infty)).
\end{align*}

We next consider the case $p=\frac34$. Define
\begin{align*}
R^{(n)}(t):=\frac{R_{A_nt}}{\sqrt{A_n}},
\qquad
\Lambda^{(n)}(t):=\frac{\sigma_{[n^t]}}{A_n},
\qquad
\phi(t):=t.
\end{align*}
By Proposition \ref{asym2}, the asymptotic behavior of $(A_n)$, and the
same monotonicity argument as above,
\begin{align*}
\lim_{n\to\infty}\sup_{t\in[0,T]}|\Lambda^{(n)}(t)-\phi(t)|=0\qquad\text{$\mathscr P_0$-a.s.}
\end{align*}
Repeating the
preceding composition argument, we obtain
\begin{align}\label{time-change-critical}
\left(
\frac{R_{\sigma_{[n^t]}}}{\sqrt{A_n}}
\right)_{0\le t\le T}
\implies
(R_t)_{0\le t\le T}
\quad\text{in }D([0,T]).
\end{align}

For $0<\delta<T$, set
\begin{align*}
\varepsilon_n
:=
\sup_{m\ge[n^\delta]}
\left|
\frac{1}{a_m\sqrt m}-1
\right|.
\end{align*}
Since $a_m\sqrt m\to1$ and $[n^\delta]\to\infty$, we have
$\varepsilon_n\to0$. Using the decomposition
\begin{align*}
\frac{\sqrt{A_n}}
{a_{[n^t]}\sqrt{n^t\log n}}
=
\sqrt{\frac{A_n}{\log n}}\cdot
\frac{1}{a_{[n^t]}\sqrt{[n^t]}}\cdot
\sqrt{\frac{[n^t]}{n^t}},
\end{align*}
we have
\begin{align*}
\sup_{t\in[\delta,T]}
\left|
\frac{\sqrt{A_n}}
{a_{[n^t]}\sqrt{n^t\log n}}
-1
\right|
\le
(1+\varepsilon_n)
\left|
\sqrt{\frac{A_n}{\log n}}-1
\right|
+\varepsilon_n+n^{-\delta}
\xrightarrow[n\to\infty]{}0.
\end{align*}
Combining this with \eqref{time-change-critical} and applying
Slutsky's theorem, we obtain
\begin{align*}
\left(
\frac{Z_{[n^t]}}{\sqrt{n^t\log n}}
\right)_{\delta\le t\le T}
\implies(R_t)_{\delta\le t\le T}
\quad\text{in }D([\delta,T]).
\end{align*}
Since $\delta>0$ and $T>0$ are arbitrary,
\begin{align*}
\left(
\frac{Z_{[n^t]}}{\sqrt{n^t\log n}}
\right)_{t>0}
\implies
(R_t)_{t>0}
\quad\text{in }D((0,\infty)).
\end{align*}
The proof is complete.
\end{proof}

We next prove the LIL in Theorem \ref{mainthm-properties-ERW}.
\begin{proof}[Proof of the LIL assertions in Theorem \ref{mainthm-properties-ERW}]
For a three-dimensional Bessel process $(R_t)$,
\begin{align}\label{BesselLIL}
\limsup_{t\to\infty}
\frac{R_t}{\sqrt{2t\log\log t}}
=1
\quad\text{$\mathscr P_0$-a.s.}
\end{align}
holds (see, e.g., \cite[Chapter IV, p. 77, No. 43]{Borodin-Salminen}).
We show that
\begin{align}\label{goal}
\limsup_{n\to\infty}
\frac{R_{\sigma_n}}
{\sqrt{2\sigma_n\log\log \sigma_n}}
=1
\quad\text{$\mathscr P_0$-a.s.}
\end{align}
First, it is clear that
\begin{align}\label{upper}
\limsup_{n\to\infty}
\frac{R_{\sigma_n}}
{\sqrt{2\sigma_n\log\log \sigma_n}}
\le1
\quad\text{$\mathscr P_0$-a.s.}
\end{align}
Next, if $\sigma_n\le t\le\sigma_{n+1}$, then
\begin{align*}
R_t\in
\left[
a_{n+1}\left(\frac{R_{\sigma_n}}{a_n}-1\right),
a_{n+1}\left(\frac{R_{\sigma_n}}{a_n}+1\right)
\right],
\end{align*}
and hence
\begin{align*}
|R_t-R_{\sigma_n}|
&\le
a_{n+1}
+
|a_n-a_{n+1}|\frac{R_{\sigma_n}}{a_n}\\
&=
a_{n+1}
+
\left|1-\frac{a_{n+1}}{a_n}\right|R_{\sigma_n}.
\end{align*}
Since
\begin{align*}
\left|1-\frac{a_{n+1}}{a_n}\right|
=
\left|1-\frac{n}{n+2p-1}\right|
\le
\frac{1}{n+2p-1},
\end{align*}
we obtain
\begin{align*}
\sup_{\sigma_n\le t\le\sigma_{n+1}}
|R_t-R_{\sigma_n}|
\le
a_{n+1}
+
\frac{1}{n+2p-1}R_{\sigma_n}.
\end{align*}
By Proposition \ref{asym2} and \eqref{an-asymp}, we have
\begin{align}\label{asym3}
\frac{a_n}{\sqrt{2\sigma_n\log\log \sigma_n}}
\sim
\begin{cases}
\displaystyle
\sqrt{\frac{3-4p}{2n\log\log n}}
& (p<\frac34),\\[3mm]
\displaystyle
\frac{1}{\sqrt{2n\log n\log\log\log n}}
& (p=\frac34),
\end{cases}
\quad\text{$\mathscr P_0$-a.s.}
\end{align}
Thus, together with \eqref{upper},
\begin{align}\label{sup}
\lim_{n\to\infty}\sup_{\sigma_n\le t\le\sigma_{n+1}}
\frac{|R_t-R_{\sigma_n}|}
{\sqrt{2\sigma_n\log\log \sigma_n}}=
0
\quad\text{$\mathscr P_0$-a.s.}
\end{align}

Next, by \eqref{BesselLIL}, there exists a sequence
$(t_m)$ with $t_m\to\infty$ such that
\begin{align}\label{lowerm}
\frac{R_{t_m}}{\sqrt{2t_m\log\log t_m}}
\ge
1-\frac{1}{m}.
\end{align}
For each $m$, choose $n_m$ such that
\begin{align*}
\sigma_{n_m}\le t_m\le\sigma_{n_m+1}.
\end{align*}
By Proposition \ref{asym2},
$\sigma_{n+1}/\sigma_n\to1$, and hence
$t_m/\sigma_{n_m}\to1$. Moreover,
\begin{align*}
\frac{R_{\sigma_{n_m}}}
{\sqrt{2\sigma_{n_m}\log\log \sigma_{n_m}}}
\ge
\frac{R_{t_m}}
{\sqrt{2\sigma_{n_m}\log\log \sigma_{n_m}}}-
\frac{|R_{t_m}-R_{\sigma_{n_m}}|}
{\sqrt{2\sigma_{n_m}\log\log \sigma_{n_m}}}.
\end{align*}
For the first term on the right-hand side, by \eqref{lowerm} and
$t_m/\sigma_{n_m}\to1$,
\begin{align*}
\liminf_{m\to\infty}
\frac{R_{t_m}}
{\sqrt{2\sigma_{n_m}\log\log \sigma_{n_m}}}=
\liminf_{m\to\infty}
\frac{R_{t_m}}{\sqrt{2t_m\log\log t_m}}
\cdot\frac{\sqrt{2t_m\log\log t_m}}
{\sqrt{2\sigma_{n_m}\log\log \sigma_{n_m}}}
\ge1.
\end{align*}
On the other hand, by \eqref{sup}, we have
\begin{align*}
\lim_{m\to\infty}\frac{|R_{t_m}-R_{\sigma_{n_m}}|}
{\sqrt{2\sigma_{n_m}\log\log \sigma_{n_m}}}
=
0.
\end{align*}
It follows that
\begin{align}\label{lower}
\limsup_{n\to\infty}
\frac{R_{\sigma_n}}
{\sqrt{2\sigma_n\log\log \sigma_n}}
\ge1
\quad\text{$\mathscr P_0$-a.s.}
\end{align}
Combining \eqref{upper} and \eqref{lower}, we obtain \eqref{goal}.
Finally, \eqref{LIL1} and \eqref{LIL2} follow from
\eqref{asym3} and Proposition \ref{Prop1.6}.
\end{proof}

A direct computation of moments from the transition probabilities
under $\mathbb Q$ does not in general lead to closed recursions.
Indeed, already for the first moment, we have
\begin{align*}
\mathbb Q[S_{n+1}]
&=
\frac{n+4p-2}{n+2p-1}
\mathbb Q[S_n]
+
\frac{n}{n+2p-1}
\mathbb Q\left[\frac{1}{S_n}\right].
\end{align*}
Thus, even the recursion for the first moment is not closed.
In contrast, the corresponding first-moment recursion for the
ordinary ERW is closed. Using Proposition \ref{asym2},
we determine the asymptotic behavior of all positive integer moments
without solving such recursions directly.
\begin{proof}[Proof of the moment assertions in Theorem \ref{mainthm-properties-ERW}]
Let $r$ be a positive integer. We first show that
\begin{align}\label{goal2}
\lim_{n\to\infty}\mathscr{P}_0\left[
\left|
\frac{R_{\sigma_n}-R_{A_n}}{\sqrt{A_n}}
\right|^r
\right]
=0.
\end{align}
Set $\widetilde{\sigma}_n:=\sigma_n\circ\Phi$. Then, we have
\begin{align*}
\mathscr{P}_0[|R_{\sigma_n}-R_{A_n}|^r]
&=
\mathbf P_{(0,0,0)}
\left[
\left|
|\mathbb{B}_{\widetilde{\sigma}_n}|
-
|\mathbb{B}_{A_n}|
\right|^r
\right]\\
&\le
\mathbf P_{(0,0,0)}
\left[
|\mathbb{B}_{\widetilde{\sigma}_n}-\mathbb{B}_{A_n}|^r
\right].
\end{align*}
Write the three-dimensional Brownian motion as
\begin{align*}
\mathbb{B}_t
=
\left(
B^{(1)}_t,B^{(2)}_t,B^{(3)}_t
\right).
\end{align*}
For each $i=1,2,3$, the process
\begin{align*}
M_t^{(i)}
:=
B^{(i)}_{t\wedge(\widetilde{\sigma}_n\wedge A_n)}
-
B^{(i)}_{t\wedge(\widetilde{\sigma}_n\vee A_n)}\qquad \text{for}\ t\ge0
\end{align*}
is a continuous martingale with
\begin{align*}
[M^{(i)}]_\infty
&=
\widetilde{\sigma}_n\vee A_n
-
\widetilde{\sigma}_n\wedge A_n\\
&=
|\widetilde{\sigma}_n-A_n|,
\end{align*}
where $[M^{(i)}]_t$ denotes the quadratic variation of $M^{(i)}$.
Therefore, applying the Burkholder--Davis--Gundy inequality
coordinatewise and using
\begin{align*}
|x|^r
\le
c_r\sum_{i=1}^3|x_i|^r,
\qquad
x=(x_1,x_2,x_3)\in\mathbb R^3,
\end{align*}
after possibly changing $c_r$, we obtain
\begin{align*}
\mathbf P_{(0,0,0)}
\left[
|\mathbb B_{\widetilde{\sigma}_n}-\mathbb B_{A_n}|^r
\right]
\le
c_r\mathbf P_{(0,0,0)}
\left[
|\widetilde{\sigma}_n-A_n|^{r/2}
\right].
\end{align*}
Hence, by Proposition \ref{asym2}, we have
\begin{align*}
\mathscr{P}_0
\left[
\left|
\frac{R_{\sigma_n}-R_{A_n}}{\sqrt{A_n}}
\right|^r
\right]
&\le
c_r\mathbf P_{(0,0,0)}
\left[
\left|
\frac{\widetilde{\sigma}_n}{A_n}-1
\right|^{r/2}
\right]\\
&=
c_r\mathscr P_0
\left[
\left|
\frac{\sigma_n}{A_n}-1
\right|^{r/2}
\right]
\xrightarrow[n\to\infty]{}0.
\end{align*}
This proves \eqref{goal2}.

By the scaling property, we have
\begin{align*}
\frac{R_{A_n}}{\sqrt{A_n}}
\overset{d}{=}R_1.
\end{align*}
Therefore, by \eqref{goal2} and the triangle inequality in
$L^r$, we have
\begin{align*}
\lim_{n\to\infty}\mathscr P_0
\left[
\left(
\frac{R_{\sigma_n}}{\sqrt{A_n}}
\right)^r
\right]
=\mathscr P_0[R_1^r].
\end{align*}
By \cite[Chapter VI, \S3, p.~251, Proposition~(3.1)]{Revuz-Yor},
the probability density function of $R_1$ is
\begin{align*}
\sqrt{\frac{2}{\pi}}x^2e^{-x^2/2}\qquad \text{for}\ x>0.
\end{align*}
By the change of variables $u=x^2/2$, we have
\begin{align*}
\mathscr P_0[R_1^r]
&=
\sqrt{\frac{2}{\pi}}
\int_0^\infty x^{r+2}e^{-x^2/2}\,dx\\
&=
\frac{2^{r/2+1}}{\sqrt{\pi}}
\Gamma\left(\frac{r+3}{2}\right).
\end{align*}
Thus, by \eqref{an-asymp} and \eqref{An-asymp}, for every positive integer $r$,
\begin{align*}
\mathbb{Q}[|S_n|^r]
&=
\frac{A_n^{r/2}}{a_n^r}
\mathscr{P}_0
\left[
\left(
\frac{R_{\sigma_n}}{\sqrt{A_n}}
\right)^r
\right]\\
&\sim
\begin{cases}
\displaystyle
\frac{2^{r/2+1}}{\sqrt{\pi}}
\Gamma\left(\frac{r+3}{2}\right)
\frac{n^{r/2}}{(3-4p)^{r/2}},
& (p<\frac34),\\[3mm]
\displaystyle
\frac{2^{r/2+1}}{\sqrt{\pi}}
\Gamma\left(\frac{r+3}{2}\right)
(n\log n)^{r/2},
& (p=\frac34).
\end{cases}
\end{align*}

Moreover, since the C-ERW is a time-inhomogeneous Markov chain and
its transition probabilities are invariant under reflection, we have
\begin{align}\label{eq-reflection-symmetry}
(S_n)_{n=0,1,2,\ldots}
\text{ under }\mathbb Q(\,\cdot\mid S_1=1)
\overset{d}{=}
(-S_n)_{n=0,1,2,\ldots}
\text{ under }\mathbb Q(\,\cdot\mid S_1=-1).
\end{align}
Hence, for every positive integer $r$,
\begin{align*}
\mathbb Q[S_n^r\mid S_1=1]
=
(-1)^r
\mathbb Q[S_n^r\mid S_1=-1].
\end{align*}
Since the C-ERW cannot change its sign without visiting the origin,
\begin{align*}
\mathbb Q[S_n^r\mid S_1=1]
&=
\mathbb Q[|S_n|^r\mid S_1=1],\\
\mathbb Q[S_n^r\mid S_1=-1]
&=
(-1)^r
\mathbb Q[|S_n|^r\mid S_1=-1].
\end{align*}
It follows that
\begin{align*}
\mathbb Q[|S_n|^r]
&=
q\mathbb Q[|S_n|^r\mid S_1=1]
+
(1-q)\mathbb Q[|S_n|^r\mid S_1=-1]\\
&=
q\mathbb Q[S_n^r\mid S_1=1]
+
(1-q)(-1)^r
\mathbb Q[S_n^r\mid S_1=-1]\\
&=
\mathbb Q[S_n^r\mid S_1=1]\\
&=(-1)^r
\mathbb Q[S_n^r\mid S_1=-1].
\end{align*}
Thus, we have
\begin{align*}
\mathbb Q[S_n^r]
&=
q\mathbb Q[S_n^r\mid S_1=1]
+
(1-q)\mathbb Q[S_n^r\mid S_1=-1]\\
&=
\bigl(q+(1-q)(-1)^r\bigr)
\mathbb Q[|S_n|^r].
\end{align*}
Distinguishing between even and odd values of $r$ completes the proof.
\end{proof}



\appendix
\section{Appendix}\label{Appendix}
Now let us determine the asymptotic behavior of the second moment of the C-ERW.
\begin{proof}[Proof of Lemma \ref{asym1}]
Let $p\le\frac34$. By \eqref{h-transform}, we have
\begin{align*}
\mathbb Q(S_{m+1}=z\mid S_0,\ldots,S_m)
=
\frac{a_{m+1}|z|}{a_m|S_m|}
1_{\{z\neq0\}}
\mathbb P(S_{m+1}=z\mid S_0,\ldots,S_m).
\end{align*}
Therefore, $\mathbb Q$-almost surely,
\begin{align*}
\mathbb Q[S_{m+1}^2\mid S_0,\ldots,S_m]
&=
\sum_{z\in\mathbb Z}
z^2\mathbb Q(S_{m+1}=z\mid S_0,\ldots,S_m)\\
&=
\sum_{k=\pm1}
(S_m+k)^2
\frac{a_{m+1}|S_m+k|}{a_m|S_m|}
1_{\{S_m+k\neq0\}}\\
&\hspace{20mm}\times
\left(
\frac12+
\left(p-\frac12\right)\frac{kS_m}{m}
\right).
\end{align*}
Hence,
\begin{align*}
\mathbb Q[S_{m+1}^2\mid S_0,\ldots,S_m]
&=
\frac{m}{m+2p-1}\frac1{|S_m|}
\sum_{k=\pm1}
(|S_m|+k)^3
\left(
\frac12+
\left(p-\frac12\right)\frac{k|S_m|}{m}
\right)\\
&=
\frac{m}{m+2p-1}
\left(
|S_m|^2+3+
\frac{2p-1}{m}(3|S_m|^2+1)
\right)\\
&=
\frac{m+6p-3}{m+2p-1}S_m^2
+
\frac{3m+2p-1}{m+2p-1}.
\end{align*}
Taking expectations, we obtain, for every $p\le\frac34$,
\begin{align}\label{second-moment-recursion}
\mathbb Q[S_{m+1}^2]
=
\frac{m+6p-3}{m+2p-1}\mathbb Q[S_m^2]
+
\frac{3m+2p-1}{m+2p-1}.
\end{align}

We first consider the case $p<\frac34$. Set
\begin{align*}
b_m
:=
\mathbb Q[S_m^2]
-
\frac{3m+2p}{3-4p}.
\end{align*}
Then $b_1=-6p/(3-4p)$ and
\begin{align*}
b_{m+1}
=
\frac{m+6p-3}{m+2p-1}b_m.
\end{align*}
It follows that
\begin{align*}
\mathbb Q[S_m^2]
=
\frac{3m+2p}{3-4p}
-
\frac{6p}{3-4p}
\prod_{k=1}^{m-1}
\left(
1+\frac{4p-2}{k+2p-1}
\right).
\end{align*}
Using $\log(1+x)=x+O(x^2)$ as $x\to0$, we obtain
\begin{align*}
\prod_{k=1}^{m-1}
\left(
1+\frac{4p-2}{k+2p-1}
\right)
=
O(m^{4p-2}).
\end{align*}
Since $4p-2<1$, we conclude that
\begin{align*}
\mathbb Q[S_m^2]
\sim
\frac{3}{3-4p}m
\qquad
(m\to\infty).
\end{align*}

We next consider the critical case $p=\frac34$. In this case,
\begin{align*}
\mathbb Q[S_{m+1}^2]
=
\frac{2m+3}{2m+1}\mathbb Q[S_m^2]
+
\frac{6m+1}{2m+1}.
\end{align*}
Set
\begin{align*}
c_m
:=
\frac{1}{2m+1}\mathbb Q[S_m^2].
\end{align*}
Then $c_1=1/3$ and
\begin{align*}
c_{m+1}
=
c_m+
\frac{6m+1}{(2m+1)(2m+3)}.
\end{align*}
Therefore, by partial fraction decomposition,
\begin{align*}
\mathbb Q[S_m^2]
&=
(2m+1)
\left(
\frac13+
\sum_{k=1}^{m-1}
\frac{6k+1}{(2k+1)(2k+3)}
\right)\\
&=
(2m+1)
\left(
3\sum_{k=0}^{m-1}\frac{1}{2k+1}-4
\right)+4.
\end{align*}
Since
\begin{align*}
\sum_{k=0}^{m-1}\frac{1}{2k+1}
\sim
\frac12\log m,
\end{align*}
we obtain
\begin{align*}
\mathbb Q[S_m^2]
\sim
3m\log m.
\end{align*}
The proof is complete.
\end{proof}


\section*{Acknowledgements}
In writing this paper, we would like to thank Professor Kouji Yano for his careful guidance and generous support. We would also like to thank Professors H\'{e}l\`{e}ne Gu\'{e}rin and Masato Takei, and Assistant Professor Kohei Hayashi, who kindly listened to our presentation on this research during their visit to The University of Osaka and gave us several valuable comments. The first author acknowledges support from JSPS KAKENHI Grant Number 26KJ1605. The second author acknowledges support from JST SPRING, Grant Number JPMJSP2138.

\bibliographystyle{plain}

\end{document}